\documentclass[a4paper, 10pt, twosides]{amsart}

\usepackage{defs}
\usepackage[font = {footnotesize, it}, labelfont = {bf, it}]{caption}
\usepackage{amsaddr}
\usepackage{color}
\usepackage{siunitx}

\title[Sparse control with intermediate constraints]{Necessary conditions for sparse optimal control problems with intermediate constraints}

\author[Y. Kumar]{Yogesh Kumar}
\address{%
	ideaForge Technology Pvt. Ltd. \\ Mahape, Navi Mumbai 400710, India.
}
\author[S. Srikant, D. Chatterjee]{Sukumar Srikant and Debasish Chatterjee}
\address{%
	Systems \& Control Engineering\\ Indian Institute of Technology Bombay, Powai\\ Mumbai 400076, India.\\\href{http://www.sc.iitb.ac.in/~srikant}{http://www.sc.iitb.ac.in/\textasciitilde{}srikant}\\\href{http://www.sc.iitb.ac.in/~srikant}{http://www.sc.iitb.ac.in/\textasciitilde{}chatterjee}
}
\author[M. Nagahara]{Masaaki Nagaraha}
\address{%
	Institute of Environmental Science and Technology\\ The University of Kitakyushu\\ Fukuoka 808-0135, Japan.\\\href{https://nagahara-masaaki.github.io}{https://nagahara-masaaki.github.io}
}

\thanks{Emails: (YK) \textsf{02yogesh16@gmail.com}; (SS) \textsf{srikant.sukumar@iitb.ac.in}; (DC) \textsf{dchatter@iitb.ac.in}; (MN) \textsf{nagahara@ieee.org}}

\date{\DTMnow}

\begin{document}
\begin{abstract}
	This article treats optimal sparse control problems with multiple constraints defined at intermediate points of the time domain. For such problems with intermediate constraints, we first establish a new Pontryagin maximum principle that provides first order necessary conditions for optimality in such problems. Then we announce and employ a new numerical algorithm to arrive at, in a computationally tractable fashion, optimal state-action trajectories from the necessary conditions given by our maximum principle. Several detailed illustrative examples are included.
\end{abstract}

   \maketitle
   
\section{Introduction}
\label{s:intro}

In this article we study a finite horizon sparse control problem with constraints on the states and control at intermediate times in addition to constraints on the states at the boundary times. There are two key ingredients in the aforementioned control problem. The first ingredient is the objective function that promotes sparsity --- it consists of an \(\Lp{0}\)-cost on the controls to design maximally sparse-in-time controllers. Such controllers are increasingly gaining prominence today, and several advantages of sparse controls or ``maximum hands-off'' controls in applications have been pointed out in \cite{ref:NagQueNes-16}. The second ingredient is the presence of intermediate constraints. The standard situation consists of boundary constraints on the pair \(\bigl( x(\bar t), x(\hat t) \bigr)\) if \([\bar t, \hat t]\) is our given interval of time. In the article at hand we define a finite sequence \(\bar t \teL t_0 < t_1 < \cdots < t_\nu \Let \hat t\) of time instants, at each of which we impose constraints on the states of the underlying system. Our results apply to a rather general class of optimal control problems that includes the design of maximally sparse controls as a special case; consequently, there are a rich set of applications, some of which are described below.

Sparse controls is an emerging area in control theory with a diverse range of applications; see, e.g., \cite{ref:NagQueNes-16, ref:ChaNagQueRao-16, ref:SriCha-16} and the references therein for a host of application areas. In particular, in networked control, sparsity is used for efficient compression and representation of control data in the form of \textit{compressive sampling} techniques, and the objective is to send such data efficiently through rate-limited communication channels such as wireless networks or the internet \cite{ref:NagMatHay-12}. In \textit{maximum hands off control} of dynamical systems, sparsity helps reduce the activation time of actuators which improves the efficiency of electric engines, in automotive industry, railway vehicles, etc.\ \cite{ref:ChaNagQueRao-16}. Control theoretic splines \cite{ref:SunEgeMar-00} that are modified to include sparsity for noise reduction and sparse representation have been proposed in \cite{ref:NagMar-14}. Since the data at multiple intermediate points are available, the target is to find an optimal interpolating curve (spline) that respects the data. This problem is recast as an optimal control problem with intermediate constraints if the interpolating curve is required to be restricted to some neighborhood of the data. Traditional control theoretic splines find applications in several areas such as trajectory planning for mobile robots, air traffic control, contour modeling of images, etc., and the inclusion of intermediate constraints to these splines vastly improves their range of applications in approximation theory and machine learning apart from the original idea of serving as motion planning primitives.

On the front of tangible and concrete applications, we start with an aircraft landing-approach problem \cite{ref:Pie-85, ref:Lygero-99}. Consider an aircraft landing-approach manoeuvre from the start of the landing phase until touch-down. For a smooth and safe landing, the aircraft must hit several frames at multiple intermediate distances from the start of the runway, and our results constitute a perfect fit in this particular problem. More generally, the class of path planning problems in the presence of obstacles \cite{ref:BlaOnoWil-11} can also be recast as a control problem with intermediate constraints. The emerging topic of control of opinion dynamics \cite{ref:EPS-17} studies the process of influencing people's opinions over a social network, where the propagation of opinions over such networks are modeled in a variety of ways. A typical problem considered in this framework is to design the campaign duration for each agent so as to ensure that certain favorable opinion levels are reached at the intermediate and final times.

There are two key contributions of this article, the first being on the theoretical front. We provide a set of first order necessary conditions for optimality in the sparse optimal control problem with intermediate constraints. To this end, a new Pontryagin maximum principle (PMP) for sparse optimal control problems with intermediate constraints is established. The techniques needed to establish this PMP derive closely from those in \cite{ref:DmiKag-11} where the authors extensively studied optimal control problems with intermediate constraints. The derivation starts by applying a transformation of the time variable, and it results in every intermediate time interval being mapped injectively to a single and fixed time interval (say $[0,1]$). This technique has been known for decades and applied to various classical variational calculus (CVC) problems see \cite{ref:CHden-37}, but it deserves to be far more widely known. In optimal control theory, this technique was applied \cite{ref:VolOst-69} for phase-constrained problems apart from the more recent work \cite{ref:DmiKag-11}, but sparsity has not been considered anywhere else. The results of \cite{ref:DmiKag-11} are \emph{not} directly applicable in our context because the problem data in \cite{ref:DmiKag-11} are smooth whereas in our setting the cost in the objective function is discontinuous in the control action variable. Indeed, maximal sparsity in time naturally involves the minimization of the so-called \(\Lp{0}\)-norm of the controller, and this particular ``norm'' can be recast as an integral with a discontinuous cost on the control actions. The standard versions of the PMP do not apply, requiring the application of a non-smooth Pontryagin maximum principle. Moreover, the requirement of maintaining desired levels of sparsity at the intermediate times further increases the complexity of the problem, requiring a careful reworking of the steps in \cite{ref:DmiKag-11} and an appeal to the nonsmooth PMP \cite[Theorem 22.26]{ref:Cla-13}. This is the content of \S\secref{s:problem formulation}-\ref{s:Proof}.

The second contribution is on the numerical front. It is well-known that, in general, indirect methods for algorithmically arriving at an optimal state-action trajectory (even corresponding to smooth data) from the necessary conditions given by the standard PMP is a difficult task. Indeed, typical algorithms rely on different variants of the Newton-Raphson shooting and homotopy methods, and deeply suffer from the lack of reasonable domains of convergence. For problems with intermediate constraints, such issues are further complicated by the fact that now the adjoint trajectories are not even continuous. Our efforts to utilize off-the-shelf solvers for our problems failed, forcing us to look for alternatives. We announce and present in \secref{s:simulations} a new `hybrid' algorithm that combines the stochastic approximation algorithm \cite{ref:Borkar-08} and Newton-Raphson iterations in a novel fashion. This particular algorithm has successfully solved all the numerical problems considered in this article where the traditional algorithms have not, and combines the best features of the traditional shooting algorithms (e.g., quadratic convergence rates) while removing the key problematic issues with them (e.g., small regions of convergence). While a detailed theoretical treatment of this algorithm will be presented elsewhere, we provide extensive details about the process of employing this new `hybrid' algorithm to arrive at optimal state-action trajectories, in a tractable fashion, from the necessary conditions given by our PMP.

\section{Problem Formulation}
\label{s:problem formulation}
In this section we formulate our problem of sparse control with intermediate constraints. Let \(\tinit < \tfin\) and consider a nonlinear dynamical system modeled by
\begin{equation}
	\label{e:plant}
	\dot{\state}(t) = \sys(t, \state(t), \cont(t)) \quad \text{for a.e.  } t \in [\tinit, \tfin],
\end{equation}
where $\state(t)\in\mathbb{R}^d$ is the vector of states and $\cont(t)\in\mathbb{R}^r$  is the control input at time $t$. We assume that $\sys:\mathbb{Q}\lra\R^d$, where the set $\mathbb{Q}\subset \Rp\times\mathbb{R}^d\times\mathbb{R}^r$ is open, is continuous and continuously differentiable with respect to the space variable $\state$ and continuous with respect to the control variable $\cont$.

For the system \eqref{e:plant}, the article \cite{ref:DmiKag-11} defined intermediate constraints by first denote a finite set of intermediate times $\tinit, \tintm{1},\ldots, \tfin$ with $\tinit < \tintm{1} <\cdots< \tfin$, that are permitted to be free. To these intermediate times, one adjoins the corresponding states to construct the vector
\begin{equation}
	\label{e:intermediate points}
	\intpt \Let \bigl((\tinit,\state(\tinit)), (\tintm{1}, \state(\tintm{1})), \ldots, (\tfin, \state(\tfin))\bigr).
\end{equation}
Intermediate constraints are described in terms of the vector \(\intpt\) in the form of the following equality and inequality constraints:
\begin{align}
	\eqcnst{j}(\intpt) &= 0\quad \text{for }j=1,2,\ldots,q,\label{e:hj constraints}\\
	\ineqcnst{i}(\intpt) &\leq 0\quad \text{for }i=1,2,\ldots,m, \label{e:gi constraints}
\end{align}
where the real-valued maps $\eqcnst{j}$ and $\ineqcnst{i}$ are defined on a given open set $\Gamma\subset (\Rp\times\mathbb{R}^d)^{\nu+1}$ and have continuous derivatives on $\Gamma$. A control $\cont$ is said to be \emph{feasible} if it satisfies the plant dynamics \eqref{e:plant} and the intermediate constraints \eqref{e:hj constraints} and \eqref{e:gi constraints} together with the action constraint
\begin{equation}
	\label{e:feasible cont}
	\cont(t) \in \admcont \quad \text{for a.e.  } t \in [\tinit, \tfin],
\end{equation}
where $\admcont$ is a given closed, compact and bounded set in $\mathbb{R}^r$.

Among the feasible controls described above, we seek a \emph{sparse optimal control} that minimizes the performance index
\begin{equation}
	\label{e:cost}
	\J(\cont) = \lambda \int_{\tinit}^{\tfin} \indic(\cont(t)) + \ell(\gamma),
\end{equation}
where $\lambda>0$ is a weight parameter, and the map \(\ell\) in the second term is a non-negative measurable cost function defined over the vector of intermediate points \(\intpt\). The integral in the first term of \eqref{e:cost} is the so-called the $L_0$ norm of the control \cite{ref:ChaNagQueRao-16} and it is the measure of the set on which $\cont$ is non-zero in the time interval $[\tinit,\tfin]$, given by
\[
   \norm{\cont}_{L_0([t_0, t_{\nu}])} \Let \text{Leb} \bigl(\{ t \in [t_0, t_{\nu}] \ | \ \cont(t) \neq 0\} \bigr);
\]
it can be written in integral form as
\[
     \norm{\cont}_{L_0([t_0, t_{\nu}])} =  \int_{\tinit}^{\tfin} \indic(\cont(t)).
\]

\begin{remark}[Intermediate constraints on the control]
In this work we include intermediate constraints on the control $\cont$ of the following form:
\begin{equation}
	\label{e:interemediate constraints sparsity}
	\lspcnst{k} \leq \int_{\tinit}^{\tintm{k}} \indic(\cont(t)) \leq \rspcnst{k}\quad \text{for }k=1,2,\ldots,\nu.
\end{equation}
To include these constraints into the problem mentioned above, we define an additional scalar state $\spstate$ that satisfies the differential equation
\begin{equation}
	\label{e:spdyn}
	\spstatedot(t) = \indic(\cont(t)),\quad \spstate(\tinit) = 0,\quad t\in[\tinit, \tfin].
\end{equation}
Then the intermediate constraints \eqref{e:interemediate constraints sparsity} can be equivalently written as
\begin{equation}
	\label{e:intermediate constraints sparsity state}
	\lspcnst{k} \leq \spstate(\tintm{k}) \leq \rspcnst{k}\quad \text{for }k=1,2,\ldots,\nu.
\end{equation}
	Despite the indicator function being discontinuous, for any measurable map \(t \mapsto \cont(t)\), the joint system
\begin{equation}
	\label{e:joint dynamics}
	\pmat{\spstatedot(t)\\ \dot\state(t)} = \pmat{\indic(\cont(t)) \\ \sys(t, \state(t), \cont(t))}
\end{equation}
	satisfies the \textit{Carath\'eodory conditions} \cite[Chapter 1]{ref:Fil-88}, and consequently, the system \eqref{e:joint dynamics} admits a Carath\'eodory solution.  Since the constraints \eqref{e:intermediate constraints sparsity state} are on the state $\spstate$, they can be included in
the inequality constraints in \eqref{e:gi constraints} via following modifications:
\begin{equation}
	\label{e:modf}
\begin{cases}
    \begin{aligned} 
       & t \mapsto \fullstate(t) \Let \pmat{ \spstate(t)\\  \state(t)} \in \R^{d+1}, \\
		& \intpt = \bigl((\tinit,\fullstate(\tinit)), (\tintm{1}, \fullstate(\tintm{1})), \ldots, (\tfin, \fullstate(\tfin))\bigr),\\
       & \eqcnst{q+1}(\intpt) \Let \spstate(\tinit) =0, \\
       & \ineqcnst{m+k}(\intpt) \Let ( \spstate(\tintm{k}) - \lspcnst{k})(\spstate(\tintm{k}) - \rspcnst{k}) \leq 0,\quad k=1,2,\ldots,\nu.\\
     \end{aligned}
\end{cases}
\end{equation}
\end{remark}

Distilling the preceding discussion, we arrive at the sparse optimal control problem with intermediate constraints:
\begin{equation}
	      \label{e:OCP} \tag{OCP}
	      \begin{aligned}
			\minimize_{\cont}	& &&  \J(\fullstate, \cont, \intpt) = \lambda \int_{\tinit}^{\tfin} \indic(\cont(t)) + \ell(\gamma).    \\
			\sbjto		        & && \begin{cases}
			                  \dot{\state}(t) = \sys(t, \state, \cont) \quad \text{for a.e.  } t \in [\tinit, \tfin], \\
			                  \spstatedot(t) = \indic(\cont(t)) \quad \text{for a.e.  } t \in [\tinit, \tfin],  \\
							   \eqcnst{j}(\intpt) = 0\quad \text{for }j=1,2,\ldots,q+1,\\
											\ineqcnst{i}(\intpt) \leq 0\quad\text{for } i=1,2,\ldots,m, \ldots, m + \nu, \\
                                            [\tinit, \tfin] \ni t \mapsto \cont(t) \in \admcont \text{ Lebesgue measurable}.
			                			\end{cases}
		\end{aligned}
\end{equation}

	\begin{remark}
		\label{r:differences}
		Two interesting features of \eqref{e:OCP} stand out: one, the integrand in the integral cost in \eqref{e:OCP} is discontinuous in \(\cont\), and two, we have constraints defined over multiple intermediate points. Since the aforementioned integrand is discontinuous in \(\cont\), the standard smooth version \cite[Chapter 4]{ref:Lib-12}, \cite[Theorem 22.13]{ref:Cla-13}, of the Pontryagin maximum principle (PMP) does not apply, and one must resort to the nonsmooth PMP \cite[Theorem 22.26]{ref:Cla-13} to account for this discontinuity. Moreover, even this particular nonsmooth PMP is inapplicable directly because of the intermediate constraints and the (possibly) free intermediate time instants. Following the techniques of \cite{ref:DmiKag-11}, we will reduce \eqref{e:OCP} to an easier and standard optimal control problem without intermediate constraints using a suitable scaling of the time variable, and then employ \cite[Theorem 22.26, Page 465]{ref:Cla-13} to provide necessary conditions for optimality in this standard problem.
	\end{remark}

    \begin{remark} \label{r:general form}
		We note that the temporal components of the vector of intermediate points \(\intpt\) are not necessarily fixed a priori. The intermediate point cost \(\ell(\intpt)\) is general enough to be of various types, e.g., it can be a function of the intermediate time instants \(t_k\), the intermediate states \(\fullstate(t_k)\), or both. Moreover, any optimal control problem in Lagrange form (i.e., integral constraints on the paths,) can be converted into its corresponding Mayer form (i.e., terminal cost) and included in \(\ell(\intpt)\) by defining a new state in a standard way. To wit, the problem \eqref{e:OCP} treated here is quite general and our main theorem below can be employed to various forms of optimal control problems, and a few special cases are described at the end of this section.
	\end{remark}
 
	\begin{definition}[Admissible process]
		The map \( t \mapsto \admps(t) \Let (\fullstate(t), \cont(t), \intpt)\) is said to be an \emph{admissible process} of the problem \eqref{e:OCP} if it satisfies all the constraints of this problem, and then we say that \(\fullstate(t)\) is the vector of admissible states corresponding to the admissible control action \(\cont(t)\) at time \(t\), and \(\intpt\) is the corresponding vector of admissible intermediate points.
	\end{definition}
             
	\begin{definition}[Local minimizer]
		\label{d:LM OCP}
		An \emph{admissible process} \(\optadmps = (\optfullstate, \optcont, \optintpt)\) is said to be a \emph{local minimizer} of \eqref{e:OCP} provided that there exists \(\epsilon > 0\) such that for every admissible process \(\admps = (\fullstate, \cont, \intpt)\) satisfying \(\unifnorm{ \fullstate - \optfullstate} \leq \epsilon\), and \(\abs{\tintm{k} - \opttintm{k}} \leq \epsilon\) for each \(k = 1, 2, \ldots, \nu\), we have \(\J(\optadmps) \leq \J(\admps)\).\footnote{The notation \(\unifnorm{\cdot}\) stands for the uniform norm; the distance induced by this norm between two controls \(u_1\) and \(u_2\) defined on \(D_1\) and \(D_2\), respectively, is given by \(\unifnorm{u_1 - u_2} = \sup_{s\in D_1\cap D_2} \norm{u_1(s) - u_2(s)}\).}
	\end{definition}

	The definition of a local minimizer in Definition \ref{d:LM OCP} is the free intermediate time-instants version of the definition of a local minimizer in \cite[Page 450]{ref:Cla-13} and it is identical to the definition of a \emph{strong minimizer} in \cite{ref:DmiKag-11}.  Our main result is the following theorem: 

	\begin{theorem}
		\label{t:OCP}
		Consider the optimal control problem \eqref{e:OCP}, and refer to the notations introduced in this section. If the process \( [\tinit,  \tfin] \ni t \mapsto \optadmps(t) = (\optfullstate(t), \optcont(t), \optintpt) \in \R^{d+1} \times \admcont \times \Gamma\) is a local minimizer of \eqref{e:OCP}, then there exist a scalar \(\eta\in\{0, 1\}\), a piecewise continuous map 
		\[
			[\optt_0, \optt_{\nu}] \ni t \mapsto \fullcostate(t) \Let \pmat{\spcostate(t) \\ \costate(t) \\ \rocostate(t)}, \quad \spcostate(t) \in \R^1, \quad \costate(t) \in \R^{d}, \quad \rocostate(t) \in \R^1,
		\]
		and multipliers \(\alpha \in \R^{q+1}, \beta \in \R^{m+\nu}\), such that with the Hamiltonian defined by
		\[
			H^{\eta}(\fullcostate, t, \fullstate, \cont) \Let  \spcostate\indic(u) + \inprod{\costate}{\sys(t, \state, \cont)} + \rocostate - \eta \lambda \indic(u), 
		\]  
		for  \(\bigl(\fullcostate, t, \fullstate, \cont \bigr) \in \R^{d+2}\times[\tinit, \tfin]\times\R^{d+1}\times\R^r\), the following conditions hold:
		\begin{enumerate}[label=\textup{(\ref{t:OCP}-\alph*)}, leftmargin=*, widest=b, align=left]
			\item \label{t:OCP:nt} non trivality: \( ( \eta, \fullcostate(t), \alpha, \beta) \neq 0 \quad  \text{for a.e.} \, t \in [\opttinit, \opttfin],\);
			\item \label{t:OCP:nn} nonnegativity: \( \beta \geq 0\);
			\item \label{t:OCP:cs} complementary slackness: \( \inprod{\beta}{\ineqc(\intpt\opt)} = 0\);
			\item \label{t:OCP:adjeq} the adjoint equations, for a.e.\ \(t \in [\opttinit, \opttfin]\), 
				\begin{equation*}
					\begin{aligned}
						& \frac{\dd \spcostate(t)}{\dd t}  = 0,\\
						& \frac{\dd \costate(t)}{\dd t} = -\Bigl(\frac{\partial \sys}{\partial \state}(t, \optstate(t), \optcont(t))\Bigr)^\top\costate(t),\\
						& \frac{\dd \rocostate(t)}{\dd t}  = -\Bigl(\frac{\partial \sys}{\partial t}(t, \optstate, \optcont(t))\Bigr)^\top\costate(t),
					\end{aligned}
				\end{equation*}    
			\item \label{t:OCP:tv} transversality conditions:
				\begin{equation*} 
					\begin{aligned}
						\begin{cases}
							\text{conditions at the endpoints of the interval:} \\
							\begin{cases}
								\spcostate(\tinit\opt) = \alpha_0, \\
								\spcostate(\tfin\opt) = - \beta_{m+\nu}(2\spstate(\tfin\opt) - \lspcnst{\nu} -\rspcnst{\nu}),  \\
								\costate(\tinit\opt) = \eta\ell_{\state(\tinit)}(\optintpt) + \big[\frac{\partial \eqc(\optintpt)}{\partial \state(\tinit)}\big]^\top\alpha + \big[\frac{\partial \ineqc(\optintpt)}{\partial \state(\tinit)}\big]^\top\beta,\\
								\costate(\tfin\opt) = - \eta\ell_{\state(\tfin)}(\optintpt) - \big[\frac{\partial \eqc(\optintpt)}{\partial \state(\tfin)}\big]^\top\alpha - \big[\frac{\partial \ineqc(\optintpt)}{\partial \state(\tfin)}\big]^\top\beta, \\
								\rocostate(\tinit\opt) = \eta\ell_{\tinit}(\optintpt) + \inprod{\alpha}{\frac{\partial \eqc(\optintpt)}{\partial \tinit}} + \inprod{\beta}{\frac{\partial \ineqc(\optintpt)}{\partial \tinit}},\\
								\rocostate(\tfin\opt) = -\eta\ell_{\tfin}(\optintpt) - \inprod{\alpha}{\frac{\partial \eqc(\optintpt)}{\partial \tfin}} - \inprod{\beta}{\frac{\partial \ineqc(\optintpt)}{\partial \tfin}}, 
							\end{cases}\\
							\text{discontinuity conditions at the intermediate points}, \, \text{for each} \, k = 1,\cdots, \nu-1, \\
							\begin{cases}
								\Delta\spcostate(\optt_k) = \spcostate(\optt_k+) - \spcostate(\optt_k-) =  \beta_{m+\nu}(2\spstate(\optt_k) - \lspcnst{k} -\rspcnst{k}),  \\			                                                         
								\Delta\costate(\optt_k) = \costate(\optt_k+) - \costate(\optt_k-) = \eta\ell_{\state(t_k)}(\optintpt) + \big[\frac{\partial \eqc(\optintpt)}{\partial \state(t_k)}\big]^\top\alpha+ \big[\frac{\partial \ineqc(\optintpt)}{\partial \state(t_k)}\big]^\top\beta , \\
								\Delta\rocostate(\optt_k) = \rocostate(\optt_k+) -  \rocostate(\optt_k-) = \eta\ell_{t_k}(\optintpt) + \inprod{\alpha}{\frac{\partial \eqc(\optintpt)}{\partial t_k}} + \inprod{\beta}{\frac{\partial \ineqc(\optintpt)}{\partial t_k}};
							\end{cases}
						\end{cases}
					\end{aligned}
				\end{equation*}  
			\item \label{t:OCP:MP} the Hamiltonian maximum:
				\[
					\optcont(t) = \underset{ u \in \admcont}\argmax\, H^{\eta}(\fullcostate, t, \optfullstate, \cont) \quad \text{ for a.e.\ } t \in [\tinit\opt, \tfin\opt];
				\]
			\item \label{t:OCP:hmc}the Hamiltonian constancy:                    
				\[
					H^{\eta}(\fullcostate, t, \optfullstate, \cont\opt) = 0 \quad \text{for a.e.\ } t \in [\tinit\opt ,\tfin\opt];
				\]        
			\item \label{t:OCP:intpt} Intermediate point:
				\[
					H^{\eta}(\fullcostate, t_k\opt+, \optfullstate, \cont\opt) - H^{\eta}(\fullcostate, t_k\opt-, \optfullstate, \cont\opt) = 0 \quad \text{for } k = 1,\ldots,\nu-1.
				\]
		\end{enumerate}
	\end{theorem}

	\begin{remark}
		Sometimes the \(L_1\)-norm over the control is employed to introduce sparsity in control, a detailed treatment of the sparsity property of \(L_1\)-optimal control problems for linear systems has been given in \cite{ref:NagQueNes-16}. It is well known that the exact \(L_0\)-optimal control problem is computationally difficult to solve when the system dynamics is non-affine in the control variable, and sometimes the \(L_1\)-optimal control problem may be employed as a surrogate for the \(L_0\)-version. In this article, apart from proving Theorem \ref{t:OCP}, we shall also introduce a new computational tool to solve the rather complicated \(L_0\)-optimal control problem with intermediate constraints.
	\end{remark}
	    
	A proof of this theorem will occupy \secref{s:Proof}. Let us briefly examine some important special cases of \eqref{e:OCP}. Consider \eqref{e:OCP} with an autonomous system model, i.e., \(\dot{\state}(t) = \sys(\state(t), \cont(t))\), and no intermediate points, i.e., \(\nu = 1\). Then \eqref{e:OCP} reduces to a standard sparse optimal control problem with free terminal time. Such problems have been investigated in detail, e.g., in the context of \emph{maximum hands-off control} \cite{ref:ChaNagQueRao-16, ref:NagQueNes-16} aimed at minimizing the controller activation time. In \cite{ref:SriCha-16} the authors proposed a jammer's perspective for \emph{sporadic denial of service} (DoS) attacks on the control signal from the perspective of sparsity. Similar problems have been investigated in the context of \emph{sparse optimal multiplexing} of linear control systems in \cite{ref:yogesh-19}, and concerns the design of sparse multiplexed controllers for an ensemble of linear systems. 


\section{Proof of the main result}
\label{s:Proof}
In this section we sketch the proof of Theorem \eqref{t:OCP} by segmenting it into multiple subsections for clarity. We first define a sparse optimal control problem and employ a nonsmooth PMP to derive necessary conditions for optimality in this problem. Then our problem \eqref{e:OCP} is transformed into this standard form using a transformation technique from \cite{ref:DmiKag-11} based on a suitable scaling of the time variable, and then we obtain necessary conditions for optimality in the aforementioned sparse optimal control problem. In order to demonstrate the applicability of the necessary conditions so derived, we show that optimality is preserved under the said transformation by establishing an equivalence between \eqref{e:OCP} and the transformed problem. In the light of this equivalence, we derive necessary conditions for \eqref{e:OCP} from the necessary conditions of the transformed problem. What makes the proof go through smoothly in this nonsmooth context is the fact that the transformation does not change \eqref{e:OCP} qualitatively, (as described in Remark \ref{r:same cost} below) but simplifies its structure.
             
\subsection{The standard problem}
\label{ss:SP}
Consider the special case of the \eqref{e:OCP} where we have an autonomous system with no intermediate points, i.e., \(\nu =1\), as discussed in Section \secref{s:problem formulation}, and let the initial time \(\tinit = 0\), and the terminal time \(\tfin = t_1\) be fixed to \(t_1 = T\), for some given real number \(T > 0\). Then \eqref{e:OCP} reduces to a standard sparse optimal control problem on a fixed time interval \( [0, T] \). The dynamics of the sparse state of this standard problem will remain the same as that of \eqref{e:OCP} while the system dynamics is given by \(\dot{\state}(t) = \sys(\state(t), \cont(t))\), and the vector of intermediate points reduces to the boundary points values,
\begin{equation}
	\label{e:SP intpt}
	\intpt  = (\fullstate(\tinit),\fullstate(T))  \in \Gamma_{sp} \subset (\R^{d+1})^2,
\end{equation} 
and the constraints over \(\intpt\) given by \eqref{e:SP intpt} reduce to:                           
\begin{equation}
	\label{e:SP constraints}
	\begin{cases}
	    \begin{aligned} 
			& \eqcnst{j}(\intpt) = 0\quad \text{for }j=1,2,\ldots,q,\\
	        & \eqcnst{q+1} =  \spstate(0) =0, \\
			& \ineqcnst{i}(\intpt) \leq 0\quad\text{for } i=1,2,\ldots,m,\\
	        & \ineqcnst{m+1}  =  ( \spstate(T) - \lspcnst{T})(\spstate(T) - \rspcnst{T}) \leq 0.
		\end{aligned}
	\end{cases}
\end{equation}
A control \(\cont\) is \emph{feasible} if it satisfies the dynamics of this problem described above, along with intermediate constraints given by \eqref{e:SP constraints}, and control constraints given by \eqref{e:feasible cont}. Consequently, the standard sparse optimal control problem can be written as
\begin{equation}
\label{e:SP} \tag{SP}
	      \begin{aligned}
			\minimize_{\cont}	& &&  \J(\fullstate, \cont, \intpt) = \lambda \int_{0}^{T} \indic(\cont(t)) + \ell(\gamma),    \\
			\sbjto		        & && \begin{cases}
			                  \dot{\state}(t) = \sys(\state(t), \cont(t)) \quad \text{for a.e. } \, t \in [0, T], \\
			                  \spstatedot(t) = \indic(\cont(t)) \quad \text{for a.e. } \, t \in [0, T],  \\
							  \text{constraints }\eqref{e:SP constraints},\\
                                            [\tinit, T] \ni t \mapsto \cont(t) \in \admcont \, \text{ Lebesgue measurable}.
			                			\end{cases}
		\end{aligned}
\end{equation}   
A local minimizer of \eqref{e:SP} satisfies the properties in Definition \ref{d:LM OCP} for \(\nu =  1\) and \(\tintm{\nu} = T\), which makes it identical to a local minimizer in the sense of \cite[Page 437]{ref:Cla-13}.

We get first order necessary conditions for optimality in \eqref{e:SP} by adapting the nonsmooth PMP \cite[Theorem 22.6, Page 465 ]{ref:Cla-13} in the following form:
\begin{theorem}
	\label{t:SP}
	Consider the optimal control problem \eqref{e:SP}, and refer to the notations introduced in this subsection and the previous section. If the process \( [0, T] \ni t \mapsto \optadmps(t) = (\optfullstate(t), \optcont(t), \optintpt) \in \R^{d+1} \times \admcont \times \Gamma_{sp}\) is a local minimizer of \eqref{e:SP}, then there exist a scalar \(\eta\in\{0, 1\}\), an absolutely continuous map 
	\[
		[0, T] \ni t \mapsto \fullcostate(t) \Let \pmat{  \spcostate(t) \\ \costate(t)} \in \R^{\sysDim+1}, \qquad \costate(t) \in \R^{\sysDim}, \qquad   \spcostate(t) \in \R,
	\]
and multipliers \(\alpha \in \R^{q+1}, \beta \in \R^{m+1}\), such that with the Hamiltonian defined by
\[
H^{\eta}(\fullcostate, \fullstate, \cont) = \inprod{\costate}{\sys(\state, \cont)} + \spcostate\indic(u) - \eta \lambda \indic(u),
\]
for \( \bigl(\fullcostate, \fullstate, \cont\bigr) \in \R^{d+1}\times\R^{d+1}\times\R^r,\)
the following conditions hold:
	\begin{enumerate}[label=\textup{(\ref{t:SP}-\alph*)}, leftmargin=*, widest=b, align=left]
			\item \label{t:SP:nt}   non triviality: \( (\eta, \fullcostate(t), \alpha, \beta) \neq 0 \quad  \text{for a.e.\ } t \in [0, T]\);
			\item \label{t:SP:nn} nonnegativity: \( \beta \geq 0 \);
			 \item \label{t:SP:cs}  complementary slackness: \(\inprod{\beta}{\ineqc(\optintpt)} = 0\);
			  \item \label{t:SP:adjeq} the adjoint equations, for a.e.\ \(t \in [0, T]\),
			             \begin{equation*}
			                  - \dot{\fullcostate}(t) = \partial_{\fullstate}H^{\eta}(\fullcostate(t), \boldsymbol\cdot, \optcont(t))(\optfullstate(t))  \implies
							  \begin{cases}
			                      \dot{\spcostate}(t) = 0,\\
								  \dot{\costate}(t) = - \Bigl(\frac{\partial\sys}{\partial\state}(\optstate(t), \optcont(t))\Bigr)^\top\costate(t);
			                  \end{cases}
			             \end{equation*}
					 \item \label{t:SP:tv} transversality conditions:\footnote{Recall that \(N_E^L(\optintpt)\) is the limiting normal cone to \(E\) (the target set of the vector of intermediate points \(\intpt\) resulting from the equality and the inequality constraints) at \(\optintpt\); see e.g., \cite[p.\ 244]{ref:Cla-13} for details.}
			                         \begin{align*}
										 & ( \fullcostate(0), - \fullcostate(T)) \in \eta\nabla\intptcost(\optintpt) + N_E^L(\optintpt)\\
										 \implies & \begin{cases}
			                                                   \spcostate(0) = \alpha_0,\\
			                                                   - \spcostate(T) = \beta_0(2{\spstate}\opt(T) - \lspcnst{T} -\rspcnst{T}),\\
			                                                   \costate(0) = \eta\ell_{\state(0)}(\optintpt) + \big[\frac{\partial \eqc(\optintpt)}{\partial \state(0)}\big]^\top\alpha + \big[\frac{\partial \ineqc(\optintpt)}{\partial \state(0)}\big]^\top\beta,\\
			                                                   - \costate(T) = \eta\ell_{\state(T)}(\optintpt) + \big[\frac{\partial \eqc(\optintpt)}{\partial \state(T)}\big]^\top\alpha+ \big[\frac{\partial \ineqc(\optintpt)}{\partial \state(T)}\big]^\top\beta;
			                                          \end{cases}
			                          \end{align*}
			   \item \label{t:SP:maxc} the Hamiltonian maximum:                    
			                        \[
			                           \optcont(t) =  \argmax_{\cont \in \admcont}H^{\eta}(\fullcostate, \optfullstate, \cont) \quad \text{for a.e.  } t \in [0 ,T];
			                        \]
			     \item \label{t:SP:hmc} the Hamiltonian constancy:                    
			                        \[
			                           H^{\eta}(\fullcostate, \optfullstate, \cont\opt) = h \quad \text{for a.e.  } t \in [0 ,T].
			                        \]      
		 \end{enumerate}
	\end{theorem}

\subsection{Transformation of \eqref{e:OCP} to the standard form}
\label{ss:TP}         
Here we transform \eqref{e:OCP} with intermediate constraints and free intermediate time instants to a problem without intermediate constraints defined on a fixed time interval. The techniques are identical to the ones in \cite{ref:DmiKag-11}, and therefore, we shall provide only the essential steps, referring the reader to \cite{ref:DmiKag-11} for complete details. The aforementioned transformation applies to the complete problem \eqref{e:OCP} with non-autonomous dynamics including the intermediate points, unlike the procedure adopted in subsection \secref{ss:SP} where we specialized \eqref{e:OCP} to the case of \(\nu = 1\) and \(\dot{\state}(t) = \sys(\state(t), \cont(t))\). \secref{ss:equivalence} and \secref{ss:characterization} contain the verification of every step of the arguments and the calculations in \cite{ref:DmiKag-11}. This verification is essential since there are certain differences in the qualitative nature of the problem data between \cite{ref:DmiKag-11} and \eqref{e:OCP}. In particular, the cost function in \eqref{e:OCP} includes a discontinuous map of the control actions, whereas the problem data in \cite{ref:DmiKag-11} are smooth.

This transformation technique is based on a suitable scaling of the time variable `\(t\)'  to another variable `\(\tau\)', where each intermediate time interval  \(\Delta_k \Let [\tintm{k-1}, \tintm{k}]\) for \(k = 1, 2, \ldots, \nu,\) is scaled to one fixed time interval \([0,1]\), and then all the states and control trajectories are transformed from the `\(t\)'-time domain to the `\(\tau\)'-time domain on each interval in the following fashion:

For each \( k = 1, 2, \ldots, \nu,\) we define an absolutely continuous map
\begin{equation}
	\label{e:scaling}
	[0, 1] \ni \tau \mapsto \ro_k(\tau) \in [ t_{k-1}, t_k] 
\end{equation}
satisfying the differential equation with a completely specified set of boundary conditions
\begin{equation}
	\label{e:scaled dynamics}
	\frac{\dd \ro_k(\tau)}{\dd \tau} = \z_k(\tau), \quad \ro_k(0) = t_{k-1}, \quad \ro_k(1) = t_{k},
\end{equation}
with \(\z_k\) acting as a new control. The function \(\ro_k\) acts as the time variable `\(t\)' on the interval \(\Delta_k\), and the values of \( \ro_k(0), \ro_k(1)\), for each \(k\) are not fixed because \(t_{k-1}\) and \(t_k\) are permitted to be free as explained in Remark \ref{r: free tf instants} ahead. In order to retain the monotonicity of time \(t\), we only consider monotone strictly increasing functions \(\ro_k\), which in turn places the restriction that \(\z_k(\tau) > 0\) for a.e.\ \(\tau \in [0, 1]\). Note that from \eqref{e:scaled dynamics} we have \(\ro_k(1) = t_{k} = \ro_{k+1}(0)\), resulting in following continuity constraints:
\begin{equation}
	\label{e:ro continuity}
	\ro_{k+1}(0) - \ro_k(1) = 0 \quad \text{for } k = 1, \ldots, \nu-1.
\end{equation}

Let us consider any admissible process \(t \mapsto \admps(t) = (\fullstate(t), \cont(t), \intpt)\) of \eqref{e:OCP}. On each interval \(t \in \Delta_k\), we transform the state \(t \mapsto \fullstate(t)\) and control \(t \mapsto \cont(t)\) trajectories by defining new maps
\begin{equation} \label{e:transformed state}
[0,1] \ni \tau \mapsto  \y_k(\tau) = \fullstate(\ro_k(\tau)) \in \R^{d+1}, \quad \text{and}
\end{equation}
\begin{equation} \label{e:transformed control}
[0,1] \ni \tau \mapsto \vcont_k(\tau) = \cont(\ro_k(\tau)) \in \R^r, 
\end{equation}

Since \(\cont(t) \in \admcont\) a.e., we have \(\vcont_k(\tau) \in \admcont\) a.e., for each $k$. Further, the continuity of the state trajectory \(t \mapsto \fullstate(t)\) on \( [\tinit, \tfin]\) along with continuity of \(\tau \mapsto \ro_k(\tau)\) on \([0,1]\) imply continuity of \(\tau \mapsto \y_k(\tau)\) on \([0, 1]\), for each \( k = 1, \ldots, \nu\).
Also, from \eqref{e:ro continuity}, we have
\begin{equation*} 
\begin{aligned} 
&  \y_k(1) = \fullstate(\ro_k(1)) = \fullstate(t_k),   
&  \y_{k+1}(0) = \fullstate(\ro_{k+1}(0)) = \fullstate(t_k),
\end{aligned}
\end{equation*}     
resulting in the following continuity constraints on \(\y_k\) at the boundary points of the interval \([0, 1]\):
\begin{equation} \label{e:continuity of ts}
\y_{k+1}(0) - \y_k(1) = 0 \quad \text{for } k = 1, \ldots, \nu-1.
\end{equation}
From \eqref{e:transformed state} we see that, \(\y_k\) can be naturally written as \(\y_k = ( \ysp_k, \yb_k^\top)^\top\), where the maps \([0, 1] \ni \tau \mapsto \ysp_k(\tau) \in \R\), and \([0, 1] \ni \tau \mapsto \yb_k(\tau) \in \R^d\) correspond to \(\spstate, \state\) respectively, and satisfy the following dynamics in view of the chain rule:
\begin{equation}  \label{e:transformed dynamics}
\begin{cases}
\begin{aligned} 
&  \frac{\dd \ysp_k(\tau)}{\dd \tau} =   \spstatedot(\rho_k(\tau)) \frac{\dd \ro_k(\tau) }{\dd \tau} =    \z_k(\tau) \indic(\vcont_k(\tau)),\\
&  \frac{\dd \yb_k(\tau)}{\dd \tau} =     \dot \state(\ro_k(\tau)) \frac{\dd \ro_k(\tau) }{\dd \tau} =    \z_k(\tau)\sys(\ro_k, \yb_k, \vcont_k).              
\end{aligned}
\end{cases}
\end{equation}     
Under the preceeding transformation $\y_k$ and $\ro_k$ play the roles of the state variables, while $\vcont_k$ and $\z_k$ act as control inputs. 

Out of the collection \( (\y_k, \ro_k, \vcont_k, \z_k)_{k=1}^{k=\nu} \) of the new state-action trajectories, we define the following maps to reduce the notational clutter:
\begin{equation} 
\begin{cases}
\begin{aligned} 
& [0 ,1] \ni \tau \mapsto \ro(\tau) \Let ( \ro_1(\tau), \ro_2(\tau), \ldots, \ro_{\nu}(\tau)) \in \R^{\nu},\\
& [0 ,1] \ni \tau \mapsto \y(\tau) \Let ( \y_1(\tau), \y_2(\tau), \ldots, \y_{\nu}(\tau)) \in (\R^{d+1})^{\nu},\\
& [0 ,1] \ni \tau \mapsto \z(\tau) \Let ( \z_1(\tau), \z_2(\tau), \ldots, \z_{\nu}(\tau)) \in \R^{\nu},\\
& [0 ,1] \ni \tau \mapsto \vcont(\tau) \Let ( \vcont_1(\tau), \vcont_2(\tau), \ldots, \vcont_{\nu}(\tau)) \in (\R^r)^{\nu},                 
\end{aligned}
\end{cases}
\end{equation}   
Subsequently, in terms of \( \ro\) and $\y$ we can write the transformed vector of intermediate points as:
\begin{equation}
\intpttd = \intpttd(\ro, \y) =   \bigl((\ro_1(0), \y_1(0)), (\ro_2(0), \y_2(0)), \cdots, (\ro_{\nu}(0), \y_{\nu}(0)), (\ro_{\nu}(1), \y_{\nu}(1))\bigr),
\end{equation}
and as a result of above transformation, it satisfies \(\intpttd = \intpt\).
The inequality constraints on the sparse state given by \eqref{e:modf} become:
\begin{equation*}
\ineqcnst{m+k} = (\ysp_k(1) - \lspcnst{k})(\ysp_k(1) - \rspcnst{k}) \leq 0 \quad \text{for each } k = 1, 2, \ldots, \nu,
\end{equation*} 
and by using the continuity constraints \eqref{e:continuity of ts}, we can rewrite the above inequality constraints and equality constraints on the sparse state as:
\begin{equation} \label{e:inequality const spstd}
\begin{cases}
\begin{aligned} 
& \eqcnst{q+1}(\intpttd) = \ysp_1(0) = 0,\\
& \ineqcnst{m+k}(\intpttd) = (\ysp_{k+1}(0) - \lspcnst{k})(\ysp_{k+1}(0) - \rspcnst{k}) \leq 0 \quad \text{for each } k = 1, 2, \ldots, \nu-1,\\
& \ineqcnst{m+\nu}(\intpttd) = (\ysp_{\nu}(1) - \lspcnst{\nu})(\ysp_{\nu}(1) - \rspcnst{\nu}) \leq 0.
\end{aligned}
\end{cases}
\end{equation} 

\begin{remark}  The number of variables increases in the transformed domain, i.e., we have multiple \(\ro_k\) and \(\y_k\). This results in an increase in the number of constraints given by the continuity constraints \eqref{e:ro continuity} and \eqref{e:continuity of ts}. Further, as mentioned in Remark \ref{r:general form} , \(t_k, \fullstate(t_k)\) are not specified and the corresponding transformed values \(\ro_k(0), \ro_k(1), \y_k(0),\) and \(\y_k(1)\) are also not specified but follow all the transformed constraints including the continuity constraints \eqref{e:ro continuity} and \eqref{e:continuity of ts}.  \label{r: free tf instants}
\end{remark}

\begin{remark}[Transformed constraints] Note that \(\intpttd = \intpt\) as asserted above, and since all the constraints of \eqref{e:OCP} are functions of \(\intpt\), the transformed constraints \(\eqcnst{j}(\intpttd)\) and \(\ineqcnst{i}(\intpttd)\) are equal to \(\eqcnst{j}(\intpt)\) and \(\ineqcnst{i}(\intpt)\) respectively, for all \( j, i,\) including the equality and the inequality constraints \eqref{e:inequality const spstd} on the sparse state. Further, since \(\intpt\) is admissible, i.e, satisfies all the constraints of \eqref{e:OCP}, \(\intpttd\) automatically satisfies all transformed constraints of \eqref{e:OCP}. Consequently, the second part of the cost function in \eqref{e:Transformed cost} satisfies \(\ell(\intpttd) = \ell(\intpt)\).    \label{r:transformed constraints}
\end{remark} 

Finally, by employing \eqref{e:scaled dynamics} and \eqref{e:transformed control} we rewrite cost function of the \eqref{e:OCP} in the new coordinates as
\begin{equation} \label{e:Transformed cost}
\tilde{\J}(\ro, \y, \z, \vcont) =  \lambda \sum_{k =1}^{\nu} \int_{0}^{1} \indic(\vcont_k(\tau))\z_k(\tau) \,\dd \tau  +  \intptcost(\tilde{\intpt}) 
\end{equation} 
This completes the transformation of the admissible process \(t \mapsto \admps(t) = (\fullstate(t), \cont(t), \intpt)\).

Let the resulting transformed process obtained above be denoted by the map:
\begin{equation} \label{e:transformed admissible process}
[0, 1] \ni \tau \mapsto \admpstd(\tau) \Let (\ro(\tau), \y(\tau), \z(\tau), \vcont(\tau)).
\end{equation}
Further, let the transformation from \(\admps\) to \(\admpstd\) described above be denoted by the map \(\F\), i.e., \(\admpstd = \F(\admps)\). Clearly, this map \(\F\) is not unique and depends on the choice of the functions \(\z_k\). If we fix these functions, this transformation becomes unique, and the resulting transformed problem is given by:
\begin{equation}
\label{e:TP} \tag{TP}
\begin{aligned}
\minimize_{\z, \vcont}	& &&      \tilde{\J}(\ro, \y, \z, \vcont) =  \lambda \sum_{k =1}^{\nu} \int_{0}^{1} \indic(\vcont_k(\tau))\z_k(\tau) \,\dd \tau  +  \intptcost(\tilde{\intpt}),\\ 
\sbjto			& && \begin{cases}
\yspdot_k = \indic(\vcont_k)\z_k \quad \text{for } k =1, 2, \ldots, \nu, \\
\dot{\yb}_k = \z_k\sys(\ro_k, \yb_k, \vcont_k) \quad \text{for } k =1, 2, \ldots, \nu, \\
\dot{\rho}_k = \z_k \quad \text{for } k =1, 2, \ldots, \nu, \\
\y_{k+1}(0) - \y_k(1) = 0 \quad \text{for } k = 1, \quad \cdots, \nu-1, \\
\ro_{k+1}(0) - \ro_k(1) = 0 \quad \text{for } k = 1, \quad \cdots, \nu-1, \\	
\eqcnst{j}(\intpttd)  = 0 \quad \text{for }  j = 1 ,2 , \ldots, q+1, \\
\ineqcnst{i}(\intpttd)  \leq 0 \quad \text{for }  i = 1 ,2 , \ldots, m+\nu,\\
[0, 1]\ni \tau \mapsto \vcont_k(\tau) \in \admcont \quad \text{for } k =1, 2, \ldots, \nu. \\   
\end{cases}
\end{aligned}
\end{equation}   

\begin{remark}
The process given by equation \eqref{e:transformed admissible process} is an admissible process of the transformed problem \eqref{e:TP}, a part of this was observed in Remark \ref{r:transformed constraints}. This claim directly follows from the definition of map \(\F\). Hence, we can construct an admissible process \(\tau \mapsto \admpstd(\tau) = (\ro(\tau), \y(\tau), \z(\tau), \vcont(\tau))\) of \eqref{e:TP} corresponding to each admissible process \(t \mapsto \admps(t) = (\fullstate(t), \cont(t), \intpt)\) of \eqref{e:OCP} via the transformation \(\admpstd = \F(\admps)\), and this \(\admpstd\) is unique provided the \(\z_k\)'s are fixed.
\end{remark}

We adapt Definition \ref{d:LM OCP} of a local minimizer for \eqref{e:TP} to the context of our problem in the following manner:
\begin{definition}
	An admissible process \(\tau \mapsto \optadmpstd = (\optro(\tau), \opty(\tau), \optz(\tau), \optvcont(\tau))\) is a \emph{local minimizer} of \eqref{e:TP} if for some \(\epsilon > 0\) and for every admissible process \(\tau\mapsto \admpstd(\tau) = (\ro(\tau), \y(\tau), \z(\tau), \vcont(\tau))\) satisfying \(\unifnorm{\ro_k - \optro_k} \leq \epsilon\) and \(\unifnorm{\y_k - \opty_k} \leq \epsilon\) for all \(k = 1,\cdots, \nu,\) we have \( \tilde{\J}(\optadmpstd) \leq  \tilde{\J}(\admpstd)\). \label{d:LM TP}
\end{definition} 

\begin{remark}
	Note that \eqref{e:TP} is defined on a fixed interval \([0, 1]\), \(\ro\) plays the role of a state variable resulting in an autonomous transformed system dynamics, \(\intpttd\) consists of just the boundary values of the states \((\ro, \y)\), and the constraints \eqref{e:ro continuity} and \eqref{e:continuity of ts} are also on the boundary points. Therefore, there are no intermediate constraints in \eqref{e:TP} unlike the problem \eqref{e:OCP}, which means it has the same structure as that of the standard problem \eqref{e:SP}. Therefore, we can directly apply Theorem \ref{t:SP} to obtain necessary conditions for optimality in \eqref{e:TP}.         
\end{remark}

\begin{theorem}
	\label{t:TP}
	Consider the optimal control problem \eqref{e:TP}, and refer to the notations introduced in this subsection. If the process \([0, 1] \ni \tau \mapsto \optadmpstd(\tau) = (\optro(\tau), \opty(\tau), \optvcont(\tau), \optz(\tau)) \in \R^{\nu}\times(\R^{d+1})^{\nu}\times(\R^r)^{\nu}\times\R^{\nu}\) is a local minimizer of \eqref{e:TP}, then there exist \(\eta\in\{0, 1\}\) and an absolutely continuous map
	\begin{align*}
		& [0, 1] \ni \tau \mapsto \fullcostatetd_k(\tau) \Let \pmat{  \spcostatetd_k(\tau) \\ \costatetd_k(\tau) \\ \rocostatetd_k(\tau) } \in \R^{\sysDim+2},\\
		& \spcostate_k(\tau) \in \R, \; \costatetd_k(\tau) \in \R^{\sysDim},\; \rocostatetd_k(\tau) \in \R, \;\text{for \( k = 1, \ldots, \nu\)},
	\end{align*}
	with \(\fullcostatetd = ( \fullcostatetd_1, \cdots, \fullcostatetd_{\nu})\), and multipliers \(\alpha \in \R^{q+1}, \beta \in \R^{m+\nu}, \lambda \in \R^{(d+1)(\nu-1)}, \delta \in \R^{\nu-1}\), such that with the Hamiltonian defined by
	\[
		\begin{aligned}
			\tilde{H}^{\eta}(\ro, \y, \vcont, \z, \fullcostatetd) = & \sum_{k=1}^{\nu}  \z_k\biggl(\spcostatetd_k\indic(\vcont_k) +\sys( \ro_k, \yb_k, \vcont_k)^\top\costatetd_k + \rocostatetd_k - \eta \lambda \indic(\vcont_k)\biggr)
		\end{aligned} 
	\]
	for \(\bigl(\ro, \y, \vcont, \z, \fullcostatetd\bigr) \in \R^{\nu}\times(\R^{d+1})^{\nu}\times(\R^r)^{\nu}\times\R^{\nu}\times(\R^{d+2})^{\nu}\), the following conditions hold:
	\begin{enumerate}[label=\textup{(\ref{t:TP}-\alph*)}, leftmargin=*, widest=b, align=left]
		\item \label{t:TP:nt} nontrivality: \( ( \eta, \fullcostatetd(\tau), \alpha, \beta, \lambda, \delta) \neq 0 \quad \text{for a.e } \tau \in [0, 1]\);
		\item \label{t:TP:nn} nonnegativity: \(\beta \geq 0\);
		\item \label{t:TP:cs} complementary slackness: \( \inprod{\beta}{\ineqc(\intpttd\opt)} = 0\);
		\item \label{t:TP:adjeq} the adjoint equation: for a.e.\ \(\tau \in [0, 1]\) and for each \(k = 1, \cdots, \nu\), we have
			\[
				\begin{aligned}
					\begin{cases}
						\dot{\spcostatetd}_k(\tau) = 0, \\
						\dot{\costatetd}_k(\tau) = - \optz_k(\tau)\Bigl(\frac{\partial\sys}{\partial\state}(\optro_k(\tau), \optyb_k(\tau), \optvcont_k(\tau))\Bigr)^\top\costatetd_k(\tau), \\	
						\dot{\rocostatetd}_k(\tau) = - \optz_k(\tau)\Bigl(\frac{\partial\sys}{\partial t}(\optro_k(\tau), \optyb_k(\tau), \optvcont_k(\tau))\Bigr)^\top\costatetd_k(\tau);
					\end{cases}
				\end{aligned}
			\]
	\item \label{t:TP:tv} transversality:
		\[
			\begin{aligned}
				& \begin{cases}
					( \fullcostatetd(0), - \fullcostatetd(1)) \in \eta\nabla\intptcost(\intpttd\opt) + N_{\Et}^L (\intpttd\opt) \implies\\
					\quad \text{for costate corresponding to the sparse state (\(\ysp_k\)):} \\
					\qquad \begin{cases}
						\spcostatetd_1(0) = \alpha_0, \\
						\spcostatetd_k(0) = \lambda^0_{k-1}\quad \text{for } k = 2,\cdots,\nu, \\
						 \spcostatetd_k(1) = \lambda^0_k - \beta_{m+k}(2\ysp_k(1) - \lspcnst{k} -\rspcnst{k}) \quad \text{for } k = 1,\cdots,\nu-1, \\
						\spcostatetd_{\nu}(1) = - \beta_{m+\nu}(2\ysp_{\nu}(1) - \lspcnst{\nu} -\rspcnst{\nu}),
					\end{cases}\\
					\quad \text{for costates corresponding to the main state (\(\yb_k\)):} \\
					\qquad \begin{cases}
						\costatetd_1(0) = \eta\ell_{\yb_1(0)}(\optintpttd) + \big[\frac{\partial \eqc(\optintpttd)}{\partial \yb_1(0)}\big]^\top\alpha + \big[\frac{\partial \ineqc(\optintpttd)}{\partial \yb_1(0)}]^\top\beta, \\
						\costatetd_k(0) = \eta\ell_{\yb_k(0)}(\optintpttd) + \big[\frac{\partial \eqc(\optintpttd)}{\partial \yb_k(0)}\big]^\top\alpha + \big[\frac{\partial \ineqc(\optintpttd)}{\partial \yb_k(0)}\big]^\top\beta + \bar{\lambda}_{k-1}\quad \text{for } k = 2,\cdots,\nu,\\
						\costatetd_k(1) = \bar{\lambda}_k\quad \text{for } k = 1,\cdots,\nu-1,\\
						\costatetd_{\nu}(1) = - \eta\ell_{\yb_{\nu}(1)}(\optintpttd) - \big[\frac{\partial \eqc(\optintpttd)}{\partial \yb_{\nu}(1)}\big]^\top\alpha - \big[\frac{\partial \ineqc(\optintpttd)}{\partial \yb_{\nu}(1)}\big]^\top\beta,\\
					\end{cases}\\
					\quad \text{for costate corresponding to time (\(t/\ro_k\)):} \\
					\qquad \begin{cases}
						\rocostatetd_1(0) = \eta\ell_{\ro_1(0)}(\optintpttd) + \inprod{\alpha}{\frac{\partial \eqc(\optintpttd)}{\partial \ro_1(0)}} + \inprod{\beta}{\frac{\partial \ineqc(\optintpttd)}{\partial \ro_1(0)}}, \\
						\rocostatetd_k(0) = \eta\ell_{\ro_k(0)}(\optintpttd) + \inprod{\alpha}{\frac{\partial \eqc(\optintpttd)}{\partial \ro_k(0)}} + \inprod{\beta}{\frac{\partial \ineqc(\optintpttd)}{\partial \ro_k(0)}} + \delta_{k-1} \quad \text{for }k = 2,\cdots,\nu,\\	
						\rocostatetd_k(1) = \delta_k\quad \text{for } k = 1,\cdots,\nu-1,\\
						\rocostatetd_{\nu}(1) = - \eta\ell_{\ro_{\nu}(1)}(\optintpttd) - \inprod{\alpha}{\frac{\partial \eqc(\optintpttd)}{\partial \ro_{\nu}(1)}} - \inprod{\beta}{\frac{\partial \ineqc(\optintpttd)}{\partial \ro_{\nu}(1)}};
					\end{cases}
				\end{cases}
			\end{aligned}
		\]
	\item \label{t:TP:MP} the Hamiltonian maximum: for a.e.\ \(\tau \in [0, 1]\),
		\begin{align*}
			& (\optvcont_k(\tau), \optz_k(\tau)) \in \argmax_{(v_k,z_k) \in \admcont \times\,]0, +\infty[} \sum_{k=1}^{\nu} \z_k(\tau)\Bigl(\spcostatetd_k(\tau)\indic(\vcont_k(\tau))\\
			& \qquad\qquad\qquad + \sys( \ro_k\opt(\tau), \yb_k\opt(\tau), \vcont_k(\tau))^\top\costatetd_k(\tau) + \rocostatetd_k(\tau) - \eta \lambda \indic(\vcont_k(\tau))\Bigr),\\
			& \optvcont_k(\tau) \in \argmax_{v_k \in \admcont} \Bigl(\spcostatetd_k(\tau)\indic(\vcont_k(\tau))\\
			& \qquad\qquad\qquad + \sys( \ro_k\opt(\tau), \yb_k\opt(\tau), \vcont_k(\tau))^\top\costatetd_k(\tau) + \rocostatetd_k(\tau) - \eta \lambda \indic(\vcont_k(\tau))\Bigr);
		\end{align*}
	\item \label{t:TP:hmc}the Hamiltonian constancy: for a.e.\ \(\tau \in [0 ,1]\),
		\[
			\spcostatetd_k(\tau)\indic(\vcont_k\opt(\tau)) + \sys( \ro_k\opt(\tau), \yb_k\opt(\tau), \vcont_k\opt(\tau))^\top\costatetd_k(\tau) + \rocostatetd_k(\tau) - \eta \lambda \indic(\vcont_k\opt(\tau)) = 0.
		\]
	\end{enumerate}
\end{theorem}

\begin{remark}
	The control variables \(\z_k\)'s enter \(\tilde{H}^{\eta}\) defined in Theorem \ref{t:TP} linearly and separately, and therefore, the \(\vcont_k\opt\)'s do not depend on the \(\z_k\opt\)'s, and the maximum condition on \(\z_k\), i.e., \(\frac{\partial\tilde{H}^{\eta}}{\partial \z_k\opt} = 0\), along with the Hamiltonain constancy condition \ref{t:SP:hmc} imply \ref{t:TP:hmc}. This says that the \(\z_k\opt\)'s play the role of scaling factors and have no impact on the optimal control actions \(\vcont_k\opt\)'s.
\end{remark}

\subsubsection*{Simplification of Theorem \eqref{t:TP}}
The non triviality condition \ref{t:TP:nt} follows the lemma:
\begin{lemma}
	\label{l:reduction of triviality condition}
	The non triviality condition \ref{t:TP:nt}, i.e., \eqref{e:nt1} below, is equivalent to condition \eqref{e:nt2}:
	\begin{align}
		\label{e:nt1} \tag{$C_1$} & (\eta, \fullcostatetd(\tau), \alpha, \beta, \lambda, \delta) \neq 0 \quad \text{for a.e } \tau \in [0, 1],\\
		\label{e:nt2} \tag{$C_2$} & (\eta, \fullcostatetd(\tau), \alpha, \beta) \neq 0 \quad \text{for a.e } \tau \in [0, 1].
	\end{align} 
\end{lemma}
\begin{proof}
	If \eqref{e:nt2} holds, then \eqref{e:nt1} obviously holds. Let us prove the implication \eqref{e:nt1} $\Rightarrow$ \eqref{e:nt2} by contradiction, so suppose that \eqref{e:nt1} hold but \eqref{e:nt2} does not. Then \((\eta, \fullcostatetd(\tau), \alpha, \beta) = 0\) for all \(\tau \in [0, 1]\). By the transversality condition \eqref{t:TP:tv} we have \(\lambda^0_k = \spcostatetd_{k+1}(0) = 0, \bar{\lambda}_k = \costatetd_{k}(1) = 0\) and \(\delta_k = \rocostatetd_k(1) = 0\), which imply \(\lambda = 0, \delta = 0\), and this violates \eqref{e:nt1}.
\end{proof}

\begin{remark}
	\label{r:reduction of transversality condition}
	Note that the multipliers \(\lambda_k\) and \(\delta_k\) appear in only a few terms of the transversality conditions, and without loss of generality, these conditions can be recast by eliminating \(\lambda_k\) and \(\delta_k\) for each \(k = 1,2, \ldots, \nu-1\), in the following way:
	\begin{equation}
		\label{e:Ntc}
		\begin{aligned}
			\begin{cases}
				\text{for costate corresponding to sparse state \((\ysp_k)\)} \\
				\qquad \spcostatetd_{k+1}(0) - \spcostatetd_k(1) =   \beta_{m+k}(2\ysp_k(1) - \lspcnst{k} -\rspcnst{k}),\\
				\text{for costates corresponding main state \((\yb_k)\)}\\
				\qquad \costatetd_{k+1}(0) - \costatetd_k(1) = \eta\ell_{\yb_{k+1}(0)}(\optintpttd) + \big[\frac{\partial \eqc(\optintpttd)}{\partial \yb_{k+1}(0)}\big]^\top\alpha + \big[\frac{\partial \ineqc(\optintpttd)}{\partial \yb_{k+1}(0)}\big]^\top\beta,\\
				\text{for costate corresponding to time \((t/\ro_k)\)}\\
				\qquad \rocostatetd_{k+1}(0) -  \rocostatetd_k(1) = \eta\ell_{\ro_{k+1}(0)}(\optintpttd) + \inprod{\alpha}{\frac{\partial \eqc(\optintpttd)}{\partial \ro_{k+1}(0)}} + \inprod{\beta}{\frac{\partial \ineqc(\optintpttd)}{\partial \ro_{k+1}(0)}}.
			\end{cases}
		\end{aligned}
	\end{equation}
\end{remark}

                                                    
\subsection{Equivalence}
\label{ss:equivalence}                     
Similar to the forward transformation map \(\F\), we define a backward transformation map $\G$ to transform any arbitrary admissible process \(\admpstd\) of \eqref{e:TP} to an admissible process of \eqref{e:OCP}. 

Firstly, let \(\tau\mapsto \admpstd(\tau) \Let (\ro(\tau), \y(\tau), \z(\tau), \vcont(\tau))\) be an arbitrary admissible process of \eqref{e:TP}. Then, by definition, \(\ro_k\) is monotone and strictly increasing on the interval \([0, 1]\) and consequently is injective. So, its inverse function exists on \(\ro_k([0, 1])\), is also monotone, strictly increasing, and bounded on the corresponding interval. The intermediate time instants are defined by
\begin{equation}
	\label{e:intermediate time instants}                   
	t_{k-1} \Let \ro_k(0) \quad \text{for } k =1, \ldots, \nu, \text{ and } t_{\nu} \Let \ro_{\nu}(1),                                       
\end{equation} 
and corresponding time intervals by \(\Delta_k \Let [t_{k-1}, t_k]\). On each interval \(\Delta_k = [t_{k-1}, t_k]\), we define the inverse map
\begin{equation}
	\label{e:inverse map}
	[t_{k-1}, t_k] \ni t \mapsto \tau(t) \Let \ro_k\inverse(t) \in [0, 1]
\end{equation}    
satisfying the dynamics:
\[
	\frac{\dd\tau(t)}{\dd t} \Let \frac{1}{z_k(\tau(t))}, \quad \tau(t_{k-1}) = 0, \quad \tau(t_{k}) = 1.
\]

Secondly, on each interval \(\Delta_k\) we transform the state \(\tau \mapsto \y_k(\tau)\) and the control \(\tau \mapsto \vcont_k(\tau)\) trajectories by defining new maps
\begin{equation}
	\label{e:G transformed state}
	[t_{k-1}, t_k] \ni t \mapsto \fullstate(t) \Let \y_k(\ro_k\inverse(t)) \in \R^{d+1}
\end{equation}
and
\begin{equation}
	\label{e: G transformed control}
	[t_{k-1}, t_k] \ni t \mapsto \cont(t) \Let \vcont_k(\ro_k\inverse(t)) \in \R^r. 
\end{equation}
The resulting intermediate states can be written as
\begin{equation}
	\label{e:g intst}
	\fullstate(t_{k-1}) = \y_k(0)\quad \text{for } k= 1, 2, \ldots, \nu, \text{ and}\quad\fullstate(t_{\nu}) = \y_{\nu}(1),
\end{equation}
with the vector of intermediate points being
\[
	\intpt(t, \fullstate) = \bigl((\tinit,\fullstate(\tinit)), (\tintm{1}, \fullstate(\tintm{1})), \ldots, (\tfin, \fullstate(\tfin))\bigr).\\
\] 
From \eqref{e:intermediate time instants} and \eqref{e:g intst} we get \(\intpt(t, \fullstate) = \intpttd(\ro, \y)\). This completes the construction of the process \(t\mapsto \admps(t) = (\fullstate(t), \cont(t), \intpt)\) and we denote this transformation by \(\admps = \G(\admtps)\). Further, for any given \(\admpstd\), the map \(\tau\mapsto \z(\tau)\) is fixed; therefore, the inverse map \eqref{e:inverse map} is unique, resulting in the uniqueness of the backward transformation map \(\G\).

Now we will establish that the process \(\admps\) obtained above is indeed an admissible process of \eqref{e:OCP}. As noted above, we have \(\intpt = \intpttd\). Therefore, on similar lines as that of Remark \ref{r:transformed constraints}, \(\intpt\) satisfies all the constraints of \eqref{e:OCP}, and since the function \(\tau \mapsto \y_k(\tau)\) is uniformly continuous on \([0, 1]\), the function \(t \mapsto \fullstate(t)\) is uniformly continuous on each interval \(\Delta_k = [\tintm{k-1}, \tintm{k}]\) and almost everywhere on \(\Delta_k\), it satisfies the differential equations
\begin{equation}
	\label{e:retransformed dynamics}
	\begin{cases}
		\begin{aligned} 
			&  \frac{\dd \spstate(t)}{\dd t} =  \frac{\dd \ysp_k(\tau)}{\dd \tau} \frac{\dd \tau(t)}{\dd t} = \indic(u(t)), \\
			&  \frac{\dd \state(t)}{\dd t} =  \frac{\dd \yb_k(\tau)}{\dd \tau} \frac{\dd \tau(t)}{\dd t} = \sys(t, \state(t), \cont(t)).             
		\end{aligned}
	\end{cases}
\end{equation}  
Continuity of \(\fullstate\) at intermediate time instants follows directly form the continuity constraints \eqref{e:continuity of ts}. Since \(\vcont_k \in \admcont\) a.e., we have \(\cont \in \admcont\) a.e. Therefore, the process \(t\mapsto \admps(t) = (\fullstate(t), \cont(t), \intpt)\) is an admissible process of \eqref{e:OCP}.

\begin{remark}
	\label{r:FGmap}
	Note that the functions \(\ro_k, \y_k, \z_k\), and \(\vcont_k\) are defined on \(\tau \in [0, 1]\) and correspond to \(\Delta_k = [t_{k-1}, t_k]\). In order to keep track of entire trajectory, we define on the interval \([0, \nu]\), corresponding to \( \Delta = [t_0, t_{\nu}]\), the following functions: for \(\s \in [k-1, k], k=1,\cdots,\nu,\)
	\begin{equation}
		\label{e:Auf}
		\begin{aligned}  
			\begin{cases}
				\auxp(\s) = \ro_k(\s - k +1) = t_{k-1} + \int_{k-1}^{\s} \z_k(\s - k + 1) \dd \s,\\
				\auxy(\s) = \y_k(\s - k +1) =  \fullstate(\ro_k(\s - k +1)), \\
				\auxv(\s) = \vcont_k(\s - k +1) =  \cont(\ro_k(\s - k +1)).
			\end{cases}
		\end{aligned}
	\end{equation}
	Since the intervals \(\Delta_k\) are concatenated consecutively, and the fact that there are continuity constraints on \(\ro_k\) \& \(\y_k\), the functions \(s\mapsto \auxp(\s)\) and \(s\mapsto \auxy(\s)\) are continuous, \(\auxp\) is strictly increasing, i.e., \( \frac{\dd \auxp(\s)}{\dd \s} \geq c > 0\) and hence inverse function exists, which is also strictly increasing in nature. We define the inverse function \( [\tinit, \tfin] \ni t \mapsto \auxpinv(t) \Let \auxp\inverse(t) \in [0, \nu] \), and for \( \z_k = \abs{\Delta_k}\), we define
	\[
		\auxpinv(t) = k-1 + \frac{t- t_{k-1}}{\abs{\Delta_k}}\quad \text{for a.e. } t \in [t_{k-1}, t_k].
	\]
	Using these functions, we can represent the process \( \admps = \G(\admtps)\) in the form  \(t\mapsto\admps(t) = (\fullstate(t), \cont(t), \intpt) \), where 
	\begin{equation}
		\label{e:Gmap}
		\fullstate(t) = \auxy(\auxpinv(t)), \quad  \cont(t) = \auxv(\auxpinv(t)).
	\end{equation}
\end{remark}

\begin{remark}
	\label{r:same cost}
	The value of the objective function remains constant under the aforementioned transformations since any admissible process of one problem is mapped to another via both the transformations \( \F \) and \( \G \), i.e., \( \tilde{\J}(\admtps) =  \J(\admps)\) for the transformation \( \admps = \G(\admtps)\) or \( \admtps = \F(\admps)\). This happens because in the transformation from \eqref{e:OCP} to \eqref{e:TP} only trajectories are being transformed from the $t$ time domain to the \(\tau\) time domain, but both problems are the same. Moreover, the map \(\G\circ\F\) is identity while \(\F\circ\G\) is not.
\end{remark}

\begin{theorem}[Equivalence]
	\label{t:equivalence}
	If the process \(t\mapsto \optadmps(t) \Let (\optfullstate(t), \optcont(t), \optintpt(t))\) is a local minimizer of \eqref{e:OCP}, then the process \(\optadmpstd = \F(\optadmps)\) is a local minimizer of the \eqref{e:TP}. Conversely, if the process \(\tau\mapsto\optadmpstd(\tau) = (\optro(\tau), \opty(\tau), \optz(\tau), \optvcont(\tau))\) is a local minimizer of \eqref{e:TP}, then the process \(\optadmps = \G(\optadmpstd)\) is a local minimizer of the \eqref{e:OCP}.
\end{theorem} 
\begin{proof}
	Here we prove only the backward implication: if the process \(\tau\mapsto\optadmpstd(\tau)\) is a local minimizer of \eqref{e:TP}, then the process \(\optadmps = \G(\optadmpstd)\) is a local minimizer of \eqref{e:OCP}; the forward implication can be proven in a similar way. From Definition \ref{d:LM OCP} and \ref{d:LM TP} of local minimizer, it suffices to show that for some \(\epsilon > 0\) and for every admissible process \(\tau\mapsto\admpstd(\tau) = (\ro(\tau), \y(\tau), \z(\tau), \vcont(\tau))\) satisfying
	\begin{equation}
		\label{c:LM constraints}
		\begin{aligned}
			\begin{cases}
				\unifnorm{\ro_k - \optro_k} \leq \epsilon, \quad \unifnorm{\y_k - \opty_k} \leq \epsilon, \quad \text{for } k = 1,\cdots, \nu,\\
				\tilde{\J}(\optadmpstd) \leq  \tilde{\J}(\admpstd),
			\end{cases}
		\end{aligned}
	\end{equation}
	the corresponding admissible process \(t\mapsto (\fullstate(t), \cont(t), \intpt(t)) = \admps(t) = \G(\admpstd(\tau))\) and the process \(t\mapsto (\optfullstate(t), \optcont(t), \optintpt(t)) = \optadmps(t) = \G(\optadmpstd(\tau))\) for some \(\epsilon_1 > 0\) satisfies:
	\[
		\unifnorm{ \fullstate - \optfullstate} \leq \epsilon_1, \quad \abs{\tintm{k} - \opttintm{k}} \leq \epsilon_1 \quad \text{for } k = 1, 2, \ldots, \nu, 
	\]
	and
	\[
		\abs{t_k - t_k^*} = \abs{\ro_k(1) -\optro_k(1)} \leq \underset{\tau \in [0, 1]}\max \abs{\ro_k(\tau) -\optro_k(\tau)} = \unifnorm{\ro_k -\optro_k} \leq \epsilon_1,
	\]
	we have \(\J(\optadmps) \leq \J(\admps)\).   

	Using Remark \ref{r:FGmap} we can construct, for a.e.\ \(t \in \Delta\) and \(\s \in [0, \nu]\), functions \(t\mapsto \auxpinv(t)\) and \(\tau\mapsto\auxy(\tau)\) and the functions \(t\mapsto\optauxpinv(t)\) and \(\tau\mapsto\optauxy(\tau)\) for each admissible process \(\admpstd\) and the process \(\optadmpstd\). Moreover, from \eqref{e:Gmap} we have \(\fullstate(t) = \auxy(\auxpinv(t))\) and \(\optfullstate(t) = \optauxy(\auxpinv(t))\). Therefore, the proof of \(\unifnorm{ \fullstate - \optfullstate} \leq \epsilon_1\) is identical to the proof of \(\unifnorm{\auxy(\auxpinv) - \optauxy(\optauxpinv)} \leq \epsilon_1.\)     

	The inequality \( \unifnorm{\y_k - \opty_k} \leq \epsilon\) on the interval \([0, 1]\) implies that \( \unifnorm{\auxy - \optauxy} \leq \epsilon\) on each interval \([k-1, k]\), and therefore on the entire interval \([0, \nu]\). Similarly, the inequality \(\unifnorm{\ro_k - \optro_k} \leq \epsilon\) implies that \(\unifnorm{\auxp - \optauxp} \leq \epsilon\) on the entire interval \([0, \nu]\); consequently, its inverse map satisfies \(\unifnorm{\auxpinv - \optauxpinv} \leq \epsilon_2(\epsilon)\) on the interval \(\Delta\cap\Delta^*\). With these we prove that \(\unifnorm{\auxy(\auxpinv) - \optauxy(\optauxpinv)} \leq \epsilon_1\) holds as follows:
	\begin{equation}
		\label{e:LM proof}
		\begin{aligned}
			\begin{cases}
				\unifnorm{\auxy(\auxpinv) - \optauxy(\optauxpinv)} & =\underset{t\in\Delta\cap\Delta^*}\max\abs{\auxy(\auxpinv(t)) - \optauxy(\optauxpinv(t))}\\
				&  \leq \underset{t\in\Delta\cap\Delta^*}\max\abs{\auxy(\auxpinv(t)) - \optauxy(\auxpinv(t))} + \underset{t\in\Delta\cap\Delta^*}\max\abs{\optauxy(\auxpinv(t)) - \optauxy(\optauxpinv(t))}\\
				&  = \unifnorm{\auxy - \optauxy}  +   \underset{t\in\Delta\cap\Delta^*}\max\abs{\optauxy(\auxpinv(t)) - \optauxy(\optauxpinv(t))}.
			\end{cases}
		\end{aligned}
	\end{equation}
	The first term on the right-hand side of \eqref{e:LM proof} is bounded by epsilon, i.e.,  \(\unifnorm{\auxy - \optauxy} \leq \epsilon\). Further, \(\optauxy\) is uniformly continuous on \([0, \nu]\), and \(\unifnorm{\auxpinv - \optauxpinv} \leq \epsilon_2(\epsilon)\) on the interval \(\Delta\cap\Delta^*\). Therefore, the second term is also bounded by some \(\delta(\epsilon_2(\epsilon)) > 0 \). Consequently,
	\begin{equation} 
		\unifnorm{\auxy(\auxpinv) - \optauxy(\optauxpinv)}  \leq \epsilon + \delta(\epsilon_2(\epsilon)) = \epsilon_1 \implies \unifnorm{ \fullstate - \optfullstate} \leq \epsilon_1.\\
	\end{equation}
	From \eqref{c:LM constraints} and Remark \ref{r:same cost} we get
	\begin{equation}
		\J(\optadmps) = \tilde{\J}(\optadmpstd) \leq \tilde{\J}(\admtps) = \J(\admps)  \quad \implies \J(\optadmps) \leq \J(\admps).
	\end{equation}
	This completes the proof.       
\end{proof} 

\subsection{Characterization of necessary conditions for \eqref{e:OCP}}
\label{ss:characterization}
In this subsection we describe the procedure to obtain the main result given by Theorem \ref{t:OCP}, using analysis done in this section \secref{s:Proof}.

Let \(\tau\mapsto\optadmpstd(\tau) \Let (\optro(\tau), \opty(\tau), \optvcont(\tau), \optz(\tau))\) be a local minimizer of \eqref{e:TP} satisfying the conditions of Theorem \ref{t:TP}. In view of Theorem \ref{t:equivalence} we construct a unique local minimizer \(\optadmps= \G(\optadmpstd) = (\optfullstate, \optcont, \optintpt)\) of \eqref{e:OCP}. Since the map \(\G\) described in \secref{ss:equivalence} simply transforms trajectories from the \(\tau\) domain to the \(t\) domain, we obtain the necessary conditions for the minimizer \(\optadmps\) by transforming the necessary conditions given by Theorem \ref{t:TP}. Moreover, Lemma \ref{l:reduction of triviality condition} along with Remark \ref{r:reduction of transversality condition} can be directly applied to get rid of the \(\lambda\) and the \(\delta\) terms. To transform the adjoint states \(\fullcostatetd\), we define, for a.e.\ \(t \in [t_{k-1}\opt, t_k\opt]\) and for each \(k = 1, 2, \ldots, \nu\), the map
\begin{equation}
	[\optt_0, \optt_{\nu}] \ni t \mapsto \fullcostate(t) = \pmat{\spcostate(t) \\ \costate(t) \\ \rocostate(t)} =  \pmat{\spcostatetd_k(\ro_k\inverse(t)) \\ \costatetd_k(\ro_k\inverse(t)) \\ \rocostatetd_k(\ro_k\inverse(t))}, \, \spcostate(t) \in \R^1, \, \rocostate(t) \in \R^{d}, \, \rocostate(t) \in \R^1,
\end{equation}
which follows, for a.e. \(t \in [\tinit\opt, \tfin\opt]\), the dynamics
\begin{equation}
	\left\{
	\begin{aligned}
		& \frac{\dd \spcostate(t)}{\dd t} = \spcostatetddot_k(\tau\opt)\frac{\dd \opttau}{\dd t} = 0\\
		& \frac{\dd \costate(t)}{\dd t} = \dot \costatetd_k\frac{\dd \opttau}{\dd t} = -[\sys_{\state}(t, \optstate, \optcont)]^\top\costate(t)\\
		& \frac{\dd \rocostate(t)}{\dd t} = \dot \rocostatetd_k(\tau\opt) \frac{\dd \opttau}{\dd t} = -[\sys_t(t, \optstate, \optcont)]^\top\costate(t).
	\end{aligned}
	\right.
\end{equation} 
Absolute continuity of \(\fullcostatetd_k\) on \([0, 1]\) implies absolute continuity of \(\fullcostate\) on each interval \( \Delta_k\opt = [t_{k-1}\opt, t_k\opt]\). Let us define the function
\[
	H^{\eta}(\fullcostate, t, \fullstate, \cont) =  \spcostate\indic(u) + \inprod{\costate}{\sys(t, \state, \cont)} + \rocostate - \eta \lambda \indic(u) 
\]
for \((p, t, x, u)\in\R^d\times\R\times\R^d\times\admcont\). Then the \textit{maximum condition} \ref{t:TP:MP} can be equivalently written as: for a.e.\ \(t \in [\tinit, \tfin]\) we have 
\[
	\optcont(t) \in \argmax_{v \in \admcont} H^{\eta}(\fullcostate(t), t, \optfullstate(t), v),
\]
and the \textit{Hamiltonian constancy condition} \ref{t:TP:hmc} becomes
\begin{equation}
	\label{e:Ham constancy}
	H^{\eta}(\fullcostate(t), t, \optfullstate(t), \cont\opt(t)) = 0 \quad \text{for a.e.\ } t \in [t_{k-1}\opt, t_k\opt] \text{ and } k = 1,\ldots,\nu,
\end{equation}
and for each \(k = 1,\ldots,\nu-1\),
\begin{equation}
	\label{e:Ham continuity}
	H^{\eta}(\fullcostate(t_k\opt), t_k\opt+, \optfullstate(t_k\opt), \cont\opt(t_k\opt+)) - H^{\eta}(\fullcostate(t_k\opt), t_k\opt-, \optfullstate(t_k\opt), \cont\opt(t_k\opt-)) = 0.
\end{equation}

\begin{remark}[Intermediate time instants]
	Note that \eqref{e:Ham constancy} implies \eqref{e:Ham continuity}. In ``hybrid'' optimal control problems where the switching times between two dynamical modes are free, \eqref{e:Ham continuity} becomes essential for the computation of the optimal switching time and is then known as the \emph{switching condition}; see e.g., \cite[Page 459]{ref:Cla-13}. Similarly in our case, the intermediate time instants \(\{t_k\}\) are free, and \eqref{e:Ham continuity} is necessary for the computation of the optimal intermediate time instants \(\{t_k\opt\}\).
\end{remark}

This completes the proof of our main result.

\section{Numerical Experiments}
\label{s:simulations}

\subsection{A new algorithm}

Theorem \ref{t:OCP} provides first order necessary conditions for optimality for \eqref{e:OCP}. In spirit, therefore, Theorem \ref{t:OCP} is similar to the classical Euler's necessary conditions for optimality (that states that the gradient of a smooth function defined on an open set must vanish at an extremum point). Numerical algorithms are thereafter needed to arrive at optimal solutions starting from the necessary conditions given by the PMP in Theorem \ref{t:OCP}.\footnote{To wit, the process of arriving at the optimizer is indirect and consequently, this method is said to be an indirect method as opposed to a direct methods (finite dimensional minimization after suitable discretization).} The process starts by applying the PMP to distill a \textit{two-point boundary value problem} (TPBVP) from \eqref{e:OCP} (see e.g., \cite{ref:Bet-98, ref:Rao-10}), following which a suitable efficient algorithm is employed to solve this TPBVP.

To be more precise, let us consider the vectors:
\begin{equation}
	\begin{aligned}
		\begin{cases}
			\gt \Let \big(\gt_0,\ldots,\gt_{\nu}\bigr) \in \R^{\nu+1},\\
			\gs \Let \big(\gs_0(\gt_0),\ldots,\gs_{\nu}(\gt_{\nu})\bigr) \in (\R^{d+1})^{\nu+1}, \\
			\ga \in \R^{q+1}, \quad \gb \in \R^{m+\nu}, \quad \gta \in \R^{m+\nu}, \\
			\gz \Let \bigl(\gt, \gs, \ga, \gb, \gta\bigr) \in \R^{\nu+1}\times(\R^{d+1})^{\nu+1}\times\R^{q+1}\times\R^{m+\nu}\times\R^{m+\nu},
		\end{cases}
	\end{aligned}
\end{equation}
where \(\gt\) is the vector of the intermediate time instants, \(\gs\) is a vector of the intermediate state vectors, \(\ga, \gb\) correspond to the multipliers \(\alpha, \beta\) defined in Theorem \ref{t:OCP}, \(\gta\) is another multiplier introduced to simplify the numerical computation by converting the inequality constraints to equality constraints, and \(\gz\) is a vector consisting of all of the aforementioned parameters. Note that \(\gz\) is the unknown vector of parameters that determines the optimal trajectories of the \eqref{e:OCP} given by Theorem \ref{t:OCP}, and therefore, solving the TPBVP consists of finding the optimal parameter \(\gz = \dst\opt\).

Let \(t\mapsto \bigl(\fullstate(t, \gz), \fullcostate(t, \gz)\bigr)\) be a Caratheodory solution of the state and adjoint dynamics given by the TPBVP in Theorem \ref{t:OCP} corresponding to some parameter \(\gz\). Consider a vector of intermediate points \(\intpt^0\) corresponding to \(\gz\), and the function
\[
	\defun: \R^{\nu+1}\times(\R^{d+1})^{\nu+1}\times\R^{q+1}\times\R^{m+\nu}\times\R^{m+\nu} \ra (\R^{d+1})^\nu\times\R^{d+2}\times\R^{q+1}\times\R^{m+\nu}\times\R^{m+\nu}\times\R^\nu
\]
defined by 
\begin{equation}
	\label{e:shooting function}
	\defun(\gz) = \pmat{
		\fullstate(\gt_1, \gz) - \gs(\gt_1) \\ 
		\vdots \\
		\fullstate(\gt_{\nu}, \gz) - \gs(\gt_{\nu}) \\
		\spcostate(\gt_{\nu}, \gz) + \gb_{m+\nu}(2\spstate(\gt_{\nu}) - \lspcnst{\nu} -\rspcnst{\nu})\\
		\costate(\gt_{\nu}, \gz)  + \eta\ell_{\state(\tfin)}(\intpt^0) + \big[\frac{\partial \eqc(\intpt^0)}{\partial \state(\tfin)}\big]^\top\ga + \big[\frac{\partial \ineqc(\intpt^0)}{\partial \state(\tfin)}\big]^\top\gb, \\ 
		\rocostate(\gt_{\nu}, \gz)  + \eta\ell_{\tfin}(\intpt^0) + \inprod{\ga}{\frac{\partial \eqc(\intpt^0)}{\partial \tfin}} + \inprod{\gb}{\frac{\partial \ineqc(\intpt^0)}{\partial \tfin}}\\
		\eqcnst\,(\intpt^0)  \\
		\ineqcnst\,(\intpt^0)  \\
		\gb_1\ineqcnst1\\
		\vdots \\
		\gb_{m+\nu}\ineqcnst{m+\nu}\\
		H^{\eta}(\fullcostate, \gt_0, \fullstate, \cont)\\
		H^{\eta}(\fullcostate, \gt_1+, \fullstate, \cont) - H^{\eta}(\fullcostate, \gt_1-, \fullstate, \cont)\\
		\vdots \\
		H^{\eta}(\fullcostate, \gt_{\nu-1}+, \fullstate, \cont) - H^{\eta}(\fullcostate, \gt_{\nu-1}-, \fullstate, \cont)
	}
\end{equation}
where \(\eqcnst\,(\intpt^0) \Let \bigl(\eqcnst1(\intpt^0),\ldots, \eqcnst{q+1}(\intpt^0))\bigr)^\top\), and \( \ineqcnst\,(\intpt^0) \Let \bigl(\ineqcnst1(\intpt^0)+(\gta_1)^2,\ldots, \ineqcnst{m+\nu}(\intpt^0)+(\gta_{m+\nu})^2\bigr)^\top\). Then our TPBVP and the function \(\defun(\cdot)\) are related by the following elementary proposition that reduces the problem of finding a solution of the \ref{e:OCP} characterized by Theorem \ref{t:OCP}, i.e., finding a \(\dst\opt\), to obtaining a zero of the \emph{nonlinear and implicit function} \(\defun(\gz)\) given by \eqref{e:shooting function}. (We note that the number of equations in \(\defun\) are equal to the number of variables in \(\gz\); therefore, this root finding problem is well posed.)

\begin{proposition}
	\label{p:TPBVP}
	The TPBVP resulting from Theorem \ref{t:OCP} for \eqref{e:OCP} has a solution \(\dst\opt = \gz\) if and only if \(\gz\) is a zero of the function \(\defun(\gz)\) given by \eqref{e:shooting function}.
\end{proposition}

The algorithms typically employed in computing a zero of a nonlinear map such as \(\defun\) above are based on the Newton-Raphson (NR) iterative scheme and continuation methods \cite{ref:Hesse-08, ref:Emn-17}. Recall that the NR iterates starts with the intention of finding a zero of the first order approximation of \(\defun\) near a zero \(\zeta\) of \(\defun\), i.e., from the affine map \(z'\mapsto \defun(z) + \defun'(z)(z' - z)\) for \(z, z'\) sufficiently close to \(\zeta\), leading to the recursion \(z_{k+1} = z_k - \defun'(z_k)\inverse\defun(z_k)\) for \(k = 0, 1, \ldots\). Under standard hypotheses the sequence \((z_k)_{k\in\N}\) of iterates converges to \(\zeta\). The effectivness of this Newton-Raphson scheme is highly dependent on
\begin{itemize}[label=\(\circ\)]
	\item the map \(\defun\) being sufficiently smooth,
	\item the availability of a good initial guess of the joint state-adjoint variables at one of the boundary points of the interval,
	\item the need for the derivative of \(\defun\) to be invertible everywhere sufficiently close to a zero of \(\defun\)), and
	\item the accuracy of the numerical computation of the derivative of \(\defun\) via finite difference schemes.
\end{itemize}
The lack, in general, of a large enough region of convergence is a well-known issue with the NR scheme. Moreover, since \(\defun\) is nonlinear and is implicitly defined in our setting, its smoothness and in turn its differentiability are difficult to ascertain a priori, resulting in difficulties with verifying the hypotheses of the NR scheme. Consequently, the need for the development of a new algorithm that is ``derivative-free'' is acute. However, if the NR scheme does converge, then it converges quadratically, which is a highly desirable property.

We propose a new recursive algorithm that combines the \emph{stochastic approximation} (SA) algorithm and the NR scheme in order to find a zero of \(\defun\).\footnote{For a detailed discussion and background of the SA algorithm we refer the reader to the standard sources \cite{ref:Borkar-08, ref:KusYin-03}.} Recall that the SA algorithm starts with the \textit{stochastic recurrence} 
\begin{equation}
	\label{e:storec}
	\dst_{k+1} = \dst_k + \stepsz_k\bigl( \defun(\dst_k) + \martsq_{k+1} \bigr),\quad \dst_0 \text{ (given)},\quad k\in\N,
\end{equation}
where \(k\) represents iteration step, \(\dst_k\) is the iterate value at the \(k^{\text{th}}\) step, \(\stepsz_k\) is a positive step size at \(k\), \(\defun(\dst_k)\) is the value of the function \(\defun\) evaluated at \(\dst_k\), and \((\martsq_k)_{k \in \N}\) is a sequence of independent and identically distributed random vectors with zero mean and bounded variance drawn from some underlying probability distribution.\footnote{This sequence \((\martsq_k)_{k\in\N}\) may be a martingale difference sequence, in general, but we did not need to employ this additional level of generality in our numerical experiments reported in this article.} Under mild hypotheses on \(\defun\), the sequence \((\dst_k)_{k \in \N}\) defined by the recursion \eqref{e:storec} asymptotically converges (in a certain precise sense) to a zero of the function \(\defun(\cdot)\) provided the sequence \((\stepsz_k)_{k\in\N}\) satisfies\footnote{This was pointed out by Robbins and Monro in \cite{ref:RobMon-51}.}
\begin{equation}
	\label{e:step condition}
	\sum_{k\in\N} \stepsz_k = +\infty \quad\text{and}\quad \sum_{k\in\N} \stepsz^2_k  < +\infty.
\end{equation}
The SA algorithm relies on the ability to evaluate \(\defun\) at given points, and
\begin{itemize}[label=\(\circ\)]
	\item neither needs the analytical expression of \(\defun\) nor computes its derivative numerically via finite difference,
	\item is consequently \emph{derivative-free}, and hence can be used even when resulting equations are non-differentiable, and
	\item does not need the availability of a good initial guess --- in fact, it explores the space on which \(\defun\) is defined due to the artificial injection of the noise, and asymptotically converges to a zero (if one exists) with probability one.
\end{itemize}
Of course, the rate of decay of the sequence \((\stepsz_k)_{k\in\N}\) (that satisfies \eqref{e:step condition}) directly affects the rate of convergence of the sequence \((\dst_k)_{k\in\N}\) despite its convergence being \emph{almost sure}.

We propose a novel ``hybrid'' algorithm here that combines some of the best features of the SA algorithm, namely, the exploration of space to find a zero, the ability to progress without derivative computations, etc., with the best features of the NR scheme, namely, a fast (quadratic) rate of convergence:
\begin{enumerate}[label=(\Roman*), leftmargin=*, align=right, widest=II]
	\item We first employ the SA algorithm to converge sufficiently close to a zero of \(\defun\); this is ensured by the difference between several successive steps of the recursions being bounded above by a sufficiently small threshold preassigned by the designer. The employment of the SA algorithm in this first step serves as an exploratory purpose as the SA algorithm finds a suitable neighborhood of a zero of \(\defun\) to settle down to, which provides a warm start for the next step.
	\item We switch to the NR scheme (or a variant thereof) with the final iterate of the SA algorithm being the initial condition of the NR iterations. The idea is that since the NR scheme typically suffers from small regions of convergence, the SA takes care of the hunt for suitable initial guesses. Since the NR iterates must converge quadratically, it becomes clear by observing very few of its iterates whether these iterates show signs of convergence. By ``signs of convergence'' we mean whether the \emph{scalars} \(\norm{\Phi(z_k)}\) and \(\norm{z_k - z_{k-1}}\) decrease on an average over several successive iterates \(k\). Moreover, higher order finite difference schemes are employed to compute the gradient matrix \(\G\) of the function \(z\mapsto \defun(z)\) since an explicit expression of \(\defun\) is not available, as is standard in shooting methods for optimal control.
		\begin{itemize}[label=\(\circ\)]
			\item If these iterates indeed converge, we continue with the NR iteration to obtain a zero.
			\item Otherwise, we simply revert back to the SA algorithm in I) above and continue with the iterations with a smaller threshold of error, and repeat until convergence.
		\end{itemize}
\end{enumerate}

A detailed theoretical treatment of this hybrid algorithm will be presented elsewhere. We mention here that the numerical examples presented below were found to range from difficult to challenging for conventional techniques based on shooting and homotopy due to the considerable sensitivity to initial guesses and the presence of discontinuities in the adjoint trajectories in our problems having intermediate constraints. However, our hybrid algorithm succeeded where others did \emph{not} in each of the numerical experiments given below.


\begin{algorithm}  
	\SetAlgoLined
	\caption{Augmented Stochastic Approximation Algorithm}
	 \label{a:Algo}  
	\textbf{initialization:} \label{a:Init2}Choose \(\dst_1\), as an initial guess, an error tolerance bound \(\epsilon > 0\), and a radius of convergence \(r > 0\) of Newton's iterates.  Set \(k = 1\), and compute \(\defun(\dst_1).\)\,       

	\While{\(\norm{\defun(\dst_k)} > r\)}{
		\begin{enumerate}[label=(\roman*), leftmargin=*, widest=b, align=left]
			\item \label{a:Algo:mart}     \( \stepsz_k\) and \(\martsq_{k+1}\)
			\item \label{a:Algo:iter}  \(\dst_{k+1} = \dst_k + \stepsz_k\bigl(\defun(\dst_k) + M_{k+1}\bigr)\)
			\item \label{a:Algo:efun}  \(\defun(\dst_{k+1})\)
			\item \label{a:Algo:count}  \( k = k + 1\)
		\end{enumerate}
	}
	\textbf{return} \(\dst_k\) and set \(\dst_1 = \dst_k\) and \(m = 1\)\\
	\While{\(\norm{\defun(\dst_m)} > \epsilon\)}{
		\begin{enumerate}[label=(\alph*), leftmargin=*, widest=b, align=left]
			\item    \(G = \partial_{\dst}\defun(\dst_m)\)
			\item   \(\dst_{m+1} = \dst_m - \alpha_m G\inverse \defun(\dst_m)\)
			\item   \(\defun(\dst_{m+1})\)
			\item   \( m = m + 1\)
		\end{enumerate}
		}
	\textbf{return} \(\dst_m\), and set \(\dst\opt = \dst_m\).           
\end{algorithm}
                    
\subsection{Numerical experiments}
       
   \begin{example} \label{Example Harmonic Oscillator}
	   Let us design a sparse controller for a linear harmonic oscillator plant described by the pair \((A, B) = \left( \pmat{0 & 1 & \\ -1 & 0 }, \pmat{ 0 \\ 1}\right)\). Let \(T > 0\) and \(\nu = 1\) (due to which \(\tfin = T\)), and consider the optimal control problem
       \begin{equation}
	      \label{e:s1} \tag{S1}
	      \begin{aligned}
			& \minimize_{\cont} && \J(\cont) =	 \int_{0}^{T} \indic(\cont(t))\,\dd t  \\
			&\sbjto			 && \begin{cases}
			                 \dot{\fullstate}(t) = A\fullstate(t) + B\cont(t) \quad \text{for a.e. } t \in [0, T], \\
			                 \fullstate(t) \in \R^2, \quad \cont(t) \in \admcont = [-1,1],\\
			                 T = 15,\\
			                 \fullstate(\tinit) = (4, -3)^{T}, \quad \fullstate(\tfin) = (0, 0)^{T}.
			\end{cases}
		\end{aligned}
		\end{equation}
         \noindent For our numerical experiment, the convergence tolerance is kept at \(\epsilon = 10^{-3}\), while the condition for switching from SA to NR is \(\norm{\defun(\gz_k)} = 0.1\), i.e., \(r =0.1\), in Algorithm \ref{a:Algo}.
        
	   The necessary conditions for an optimal control described by Theorem \ref{t:OCP} and Algorithm \ref{a:Algo} along with the idea of Proposition \ref{p:TPBVP} are applied for numerical computation of the solution trajectories of this system; the corresponding results are demonstrated in Figures \ref{fig:fig09} through \ref{fig:fig12}. Figure \ref{fig:fig09} plots the convergence of \(\norm{\defun(\gz_k)}\) w.r.t.\ the iteration number \(k\) for three different sequences \((\gamma_k)\) of the step size:
           \[
           \gamma_k = \frac{10^{-2}}{ k^{\frac{4}{7}}},\quad \gamma_k = \frac{10^{-2}}{1 + 0.05k}, \quad \text{and  } \gamma_k = \frac{10^{-2}}{1 + k^{\frac{1}{5}}}.        
           \]
		   It is clear from Figure \ref{fig:fig09} that the last sequence gives the fastest convergence among the three: the number of iterations required to converge to \(10^{-1}, 10^{-2}, \&  10^{-3}\) are approximately \( 350, 500, \& 4000\) in comparison to \(7000, 10000, \& 14000\), and \(3500, 9000, \& 35000+\), respectively for the other two sequences. The average CPU time required per iteration of SA algorithm is around \(16.5\) milli-seconds and is not dependent on the sequence \((\gamma_k)\).

    \begin{remark}[Necessity for switching] As illustrated in Figures \ref{fig:fig09} and \ref{fig:fig10} of Example \ref{Example Harmonic Oscillator}, the SA algorithm alone has the ability to reach sufficiently close to a zero of \(\norm{\defun(\cdot)}\). However, when the iterates reach close to a zero, then the future values of the noise \(M_n\) in the Algorithm \ref{a:Algo} become comparable to that of \(\defun(\gz_n)\), which results in sustained oscillations that hinder further convergence of the SA. This phenomenon can be seen in Figure \ref{fig:fig09} where all three trajectories exhibit oscillatory behavior around \(\norm{\defun(\cdot)} = 0.1\). To mitigate this difficulty with convergence, we switch to the NR method at such a stage since it gives quick quadratic convergence (depicted in Figure \ref{fig:fig10}) when the initial value lies within its region of convergence. Apart from this, there are many sequences, e.g., the sequence \(\gamma_k = \frac{10^{-2}}{1 + 0.05k}\), which take a long time to converge to a low value \(\norm{\defun(\cdot)} = 0.001\), but yield sufficiently faster rate of convergence at the initial stage. This feature provides ample justification for the switching in our algorithm. \end{remark}
   
		\textup{Figure \ref{fig:fig10} depicts the convergence of the sequence \(\bigl(\norm{\defun(\gz_k)}\bigr)_k\) obtained by application of the Augmented Stochastic Approximation Algorithm \ref{a:Algo}. The SA performs well and gives smooth convergence initially. The NR-based shooting takes around $174$ seconds and $2200$ iterations to converge to  \(\norm{\defun(\cdot)} = 0.001 \text{ from } 0.1\). The maximum time taken by the Algorithm \ref{a:Algo} to converge to \(\norm{\defun(\cdot)} = 0.001\) is \(232\) seconds, and is much less than the maximum time required for the same convergence using only the SA algorithm which is more than 580 seconds for this particular example.}
    
    \begin{remark} Note that the sequence \(\gamma_k = \frac{10^{-2}}{1 + k^{1/5}}\) performs well in this example, but its square sum is not bounded, i.e., it does not satisfy second part of the \eqref{e:step condition}. While convergence is not theoretically guaranteed, its transient performance has been empirically found to be sufficiently good in these difficult problems to merit further investigation.
    \end{remark}
    
    \textup{ In Figure \ref{fig:fig11} we plot the optimal control trajectory of the harmonic oscillator; it exhibits \emph{bang-off-bang} profile and is periodic in nature with time period \(TP =2\pi\) seconds. Also, the controller is set to zero for most part of its activation period meaning the sparsity is playing its role. Figure \ref{fig:fig12} shows the time evolution of the optimal state trajectories; sharpe changes in trajectory \(x_2\opt\) can be seen when ever the control \(\cont\opt(t)\) switches. Finally, the trajectories reach origin at \(t = 14 s\) and stays there afterwards.}
   \end{example}

        \begin{figure}[h!]
     \centering
             \begin{minipage}{0.48\textwidth}
             \centering
                   \captionsetup{width=0.9\linewidth}
                    \includegraphics[width=1.1\linewidth]{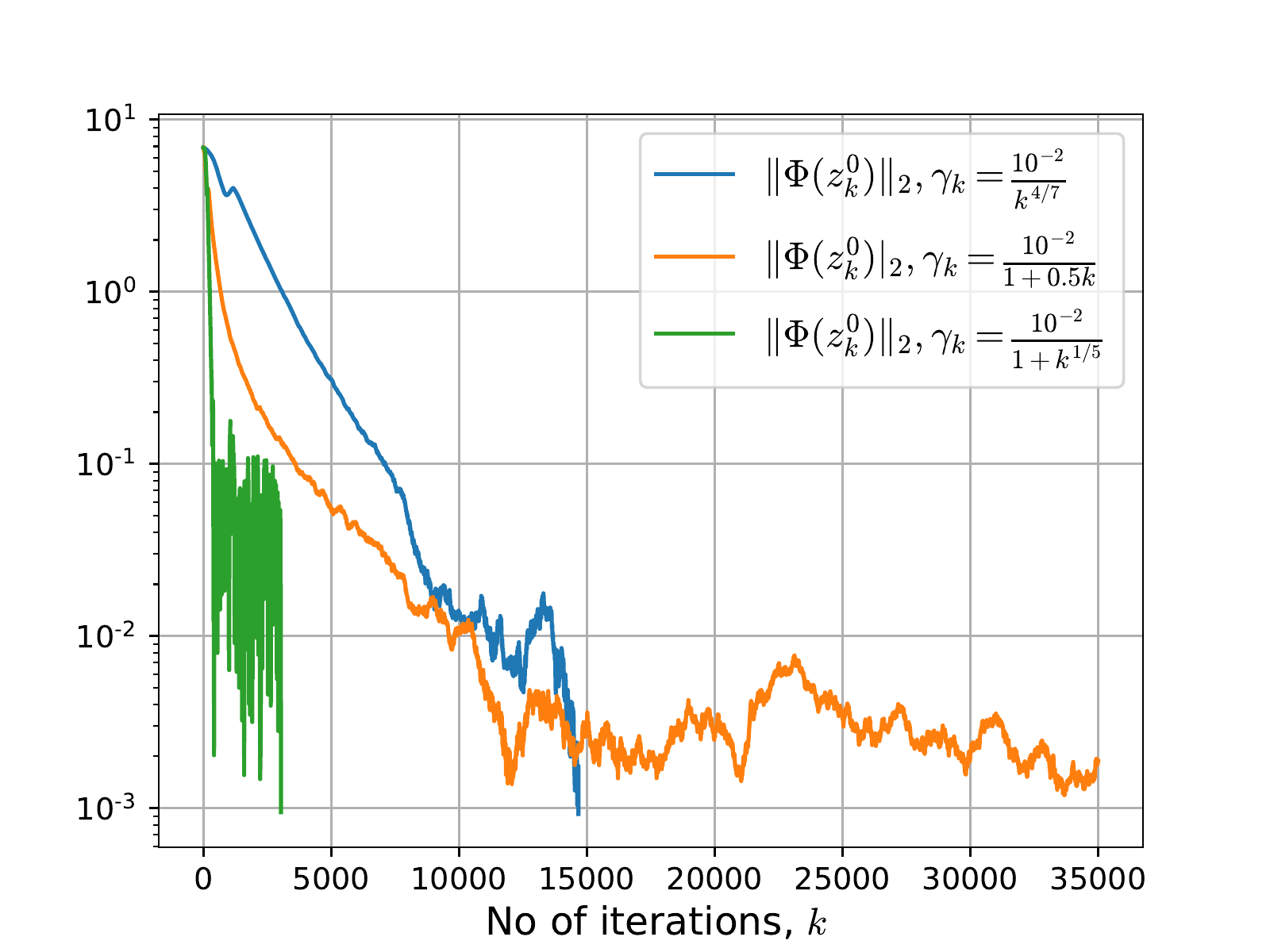}
                    \vspace{-20pt}
                    \caption{Illustration of the convergence of the norm \(\norm{\defun(\gz_k)}\) with respect to the number of iterations \(k\), obtained using only SA part of the Algorithm \ref{a:Algo}, for three different form of the step size \(\gamma_k\). }
                    \label{fig:fig09}
            \end{minipage}
            \begin{minipage}{0.48\textwidth}
            \centering
                     \captionsetup{width=0.9\linewidth}
                     \includegraphics[width=1.1\linewidth]{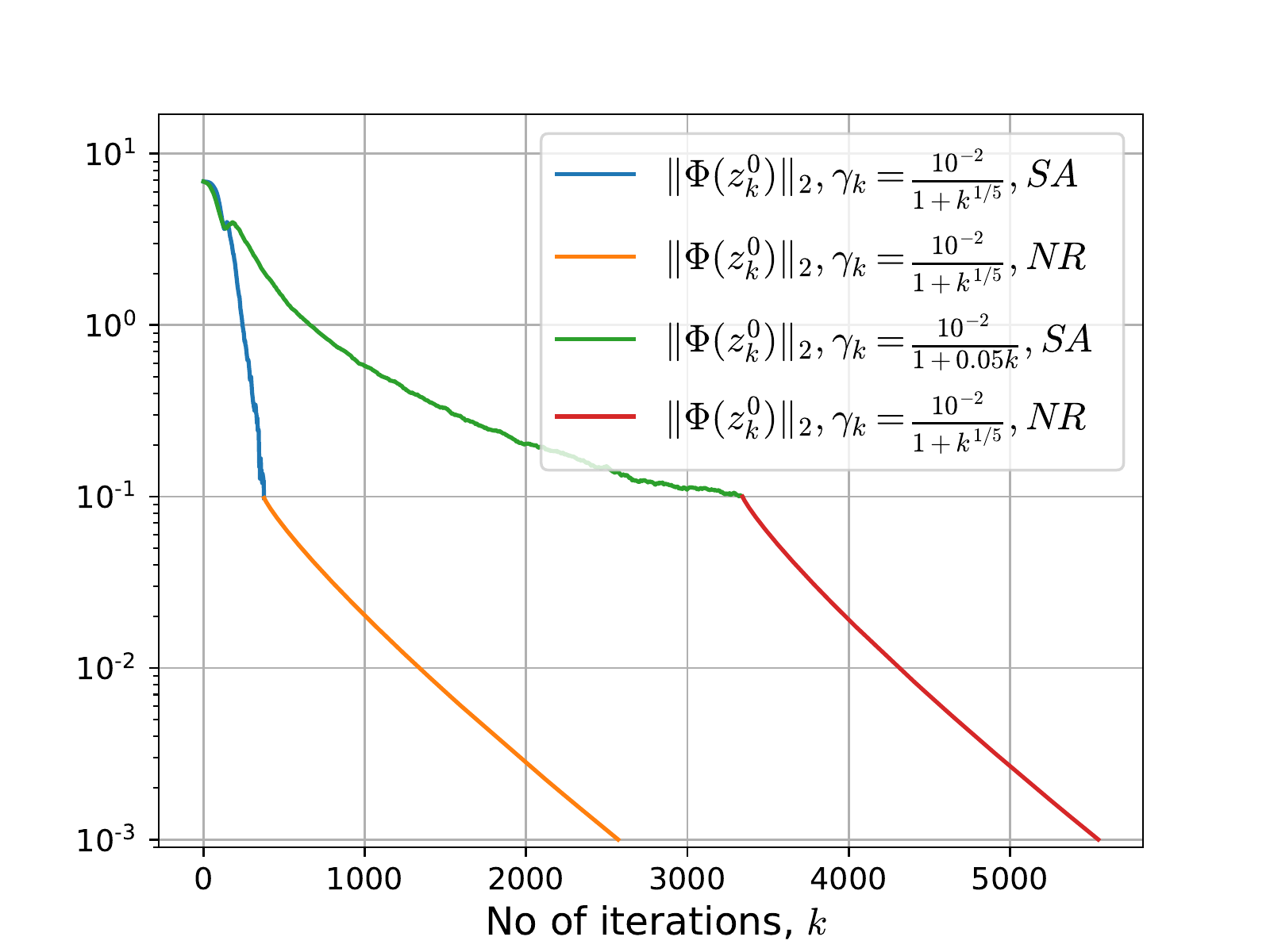}
                       \vspace{-20pt}\caption{We plot the convergence of the norm \(\norm{\defun(\gz_k)}\) w.r.t the number of iterations \(k\), obtained from the Augmented Algorithm \ref{a:Algo}; Overall time taken for convergence to \(0.001\) is $180$ seconds for the faster one and $232$ seconds for the slower one.}                             
                     \label{fig:fig10}
            \end{minipage}
       \end{figure}
       
        \begin{figure}[h!]
     \centering
             \begin{minipage}{0.48\textwidth}
             \centering
                   \captionsetup{width=0.9\linewidth}
                    \includegraphics[width=1.1\linewidth]{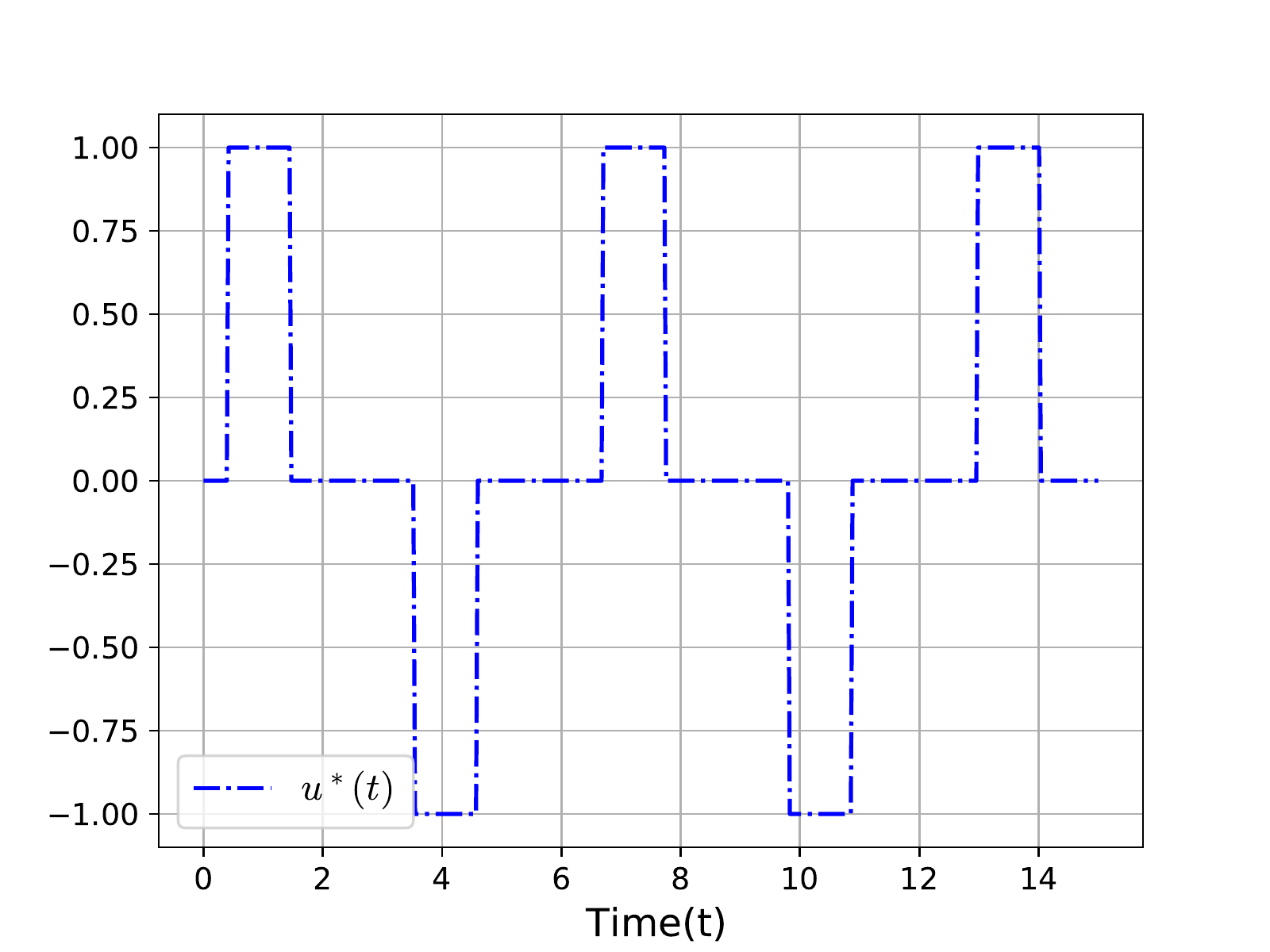}
                    \vspace{-20pt}
                    \caption{Plots the time evolution of the optimal control trajectory \(\cont\opt(t)\); the trajectory is periodic in nature with time period of \(2\pi\) seconds, is very time sparse, and has \emph{bang-off-bang} profile.} 
                    \label{fig:fig11}
            \end{minipage}
            \begin{minipage}{0.48\textwidth}
            \centering
                     \captionsetup{width=0.9\linewidth}
                     \includegraphics[width=1.1\linewidth]{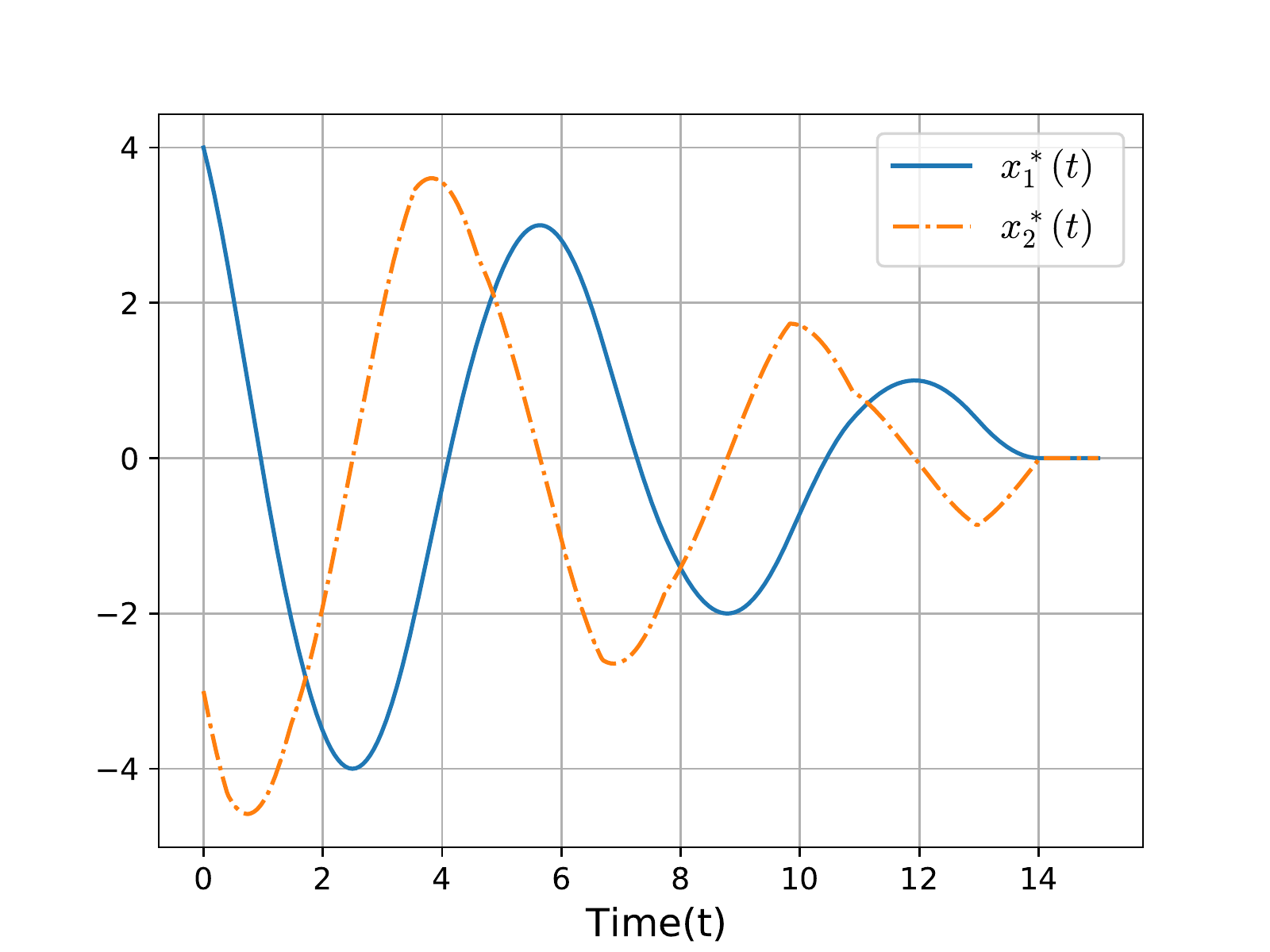}
                       \vspace{-20pt}
                      \caption{Shows the time evolution of the optimal state trajectories; sharpe changes in trajectory \(x_2\opt(\cdot)\) can be seen when ever the control \(\cont\opt(t)\) switches.}
                     \label{fig:fig12}
            \end{minipage}
       \end{figure}

\begin{example}   
	   In this example we illustrate the intermediate constraints problem on the evolution of opinions on a 3-regular graph with number of edges (agent) \(d = 24\) and number of influence channel \(r = 2\). Let \(x_i(t)\) denote the opinion of the \(i\)th agent, \(a_{ij}\) denote the magnitude of effect of opinion of \(j\)th agent on the opinion of \(i\)th agent, \(\cont_k\) denote the \(k\)th influence channel, and \(b_{ik}\) denote the influence of utilization of channel \(k\) on opinion of agent \(i\). In these terms the opinion dynamics of agent \(i\) is given by 
	\[
		\dot x_i(t) \Let \underset{j\in\{1,\ldots,24\}}\sum a_{ij}(x_j(t) - x_i(t)) + \underset{k\in\{1,2\}}\sum b_{ik}\cont_k(t)
	\]
	where \(a_{ij} \geq 0\). By stacking the dynamics of all the agents we arrive at a linear dynamical system. The exact optimization problem is given by:
	\begin{equation}
		\label{s:eg2} \tag{S2}
		\begin{aligned}
			\minimize_{\cont}  & && \J(\fullstate,\cont,\intpt) = \int_{0}^{T} \indic(\cont(t))\,\dd t - \inprod{p}{x(T)}  \\
			\sbjto				& && \begin{cases}
			\dot{\fullstate}(t) = -(D-A)\fullstate(t) + B\cont(t) \quad \text{for a.e. } t \in [\tinit, \tfin = T], \\		      
			\fullstate(t) \in \R^d, \quad \cont(t) \in [-1,1]\times[-1,1],\\
			t_0 = 0s,\quad t_1 = 2s,\quad T= 4s,
		\end{cases}
		\end{aligned}
	\end{equation}   
	where $A$ is the adjacency matrix and $D$ is the degree matrix of the graph, $B$ is the channel influence matrix, $p$ is a weight vector and the intermediate states \(x(t_0), x(t_1)\) are fixed while the terminal state \(x(T)\) is kept free, see \cite{ref:EPS-17} for further discussion on such problems. 

	\textup{The optimal control described by Theorem \ref{t:OCP} and Algorithm \ref{a:Algo} along with the idea of Proposition \ref{p:TPBVP} are applied to the problem \eqref{s:eg2} and corresponding results are shown in Figures \ref{fig:fig06} through \ref{fig:fig08}. Figure \ref{fig:fig06} plots the evolution of opinions: each trajectory starts from some randomly chosen opinions between \(]-1,1[\) at the initial time \(t_0 = \SI{0}{s}\), and reach exactly the three different specified levels (0.6, 0.65, 0.75) of the opinions (depending on the influence matrix $B$) at the intermediate time \(t_1\). Optimality requirements further push the opinions to higher values at the terminal time \(T\). Sharp changes in the solution trajectories can be seen whenever the corresponding controllers \(u^*_1\) or \(u^*_2\) switch.}

	\textup{In Figure \ref{fig:fig07} we depict the optimal control trajectories \( u^*_1\) and \(u^*_2\): both trajectories show \textit{bang-off-bang} nature in accordance with \cite{ref:ChaNagQueRao-16} and each of them remains zero for some part of their activation time; this shows that the sparsity is playing its role. Also on the first interval, \(u^*_1\) is active for smaller time span than \(u^*_2\); this happens because the corresponding opinions are required to reach a lower level at time \(t_1\), and has negative influence on the opinion of others. Figure \ref{fig:fig08} shows the evolution of the norm of all the states: Initially the norm decreases for some time showing a non-minimum phase behaviour and then increases with varying rate till around \(t = \SI{1.6}{s}\). On the interval \( t \in [1.6, 1.8] \) it decreases slightly because both the controllers are switched off, following which it increases again until around \( t = \SI{3.3}{s}\) and finally reaches close to 5.1 units.} 
\end{example}

    \begin{figure}[h]
     \centering
             \begin{minipage}{0.48\textwidth}
             \centering
                    \captionsetup{width=0.9\linewidth}
                    \includegraphics[width=1.1\linewidth]{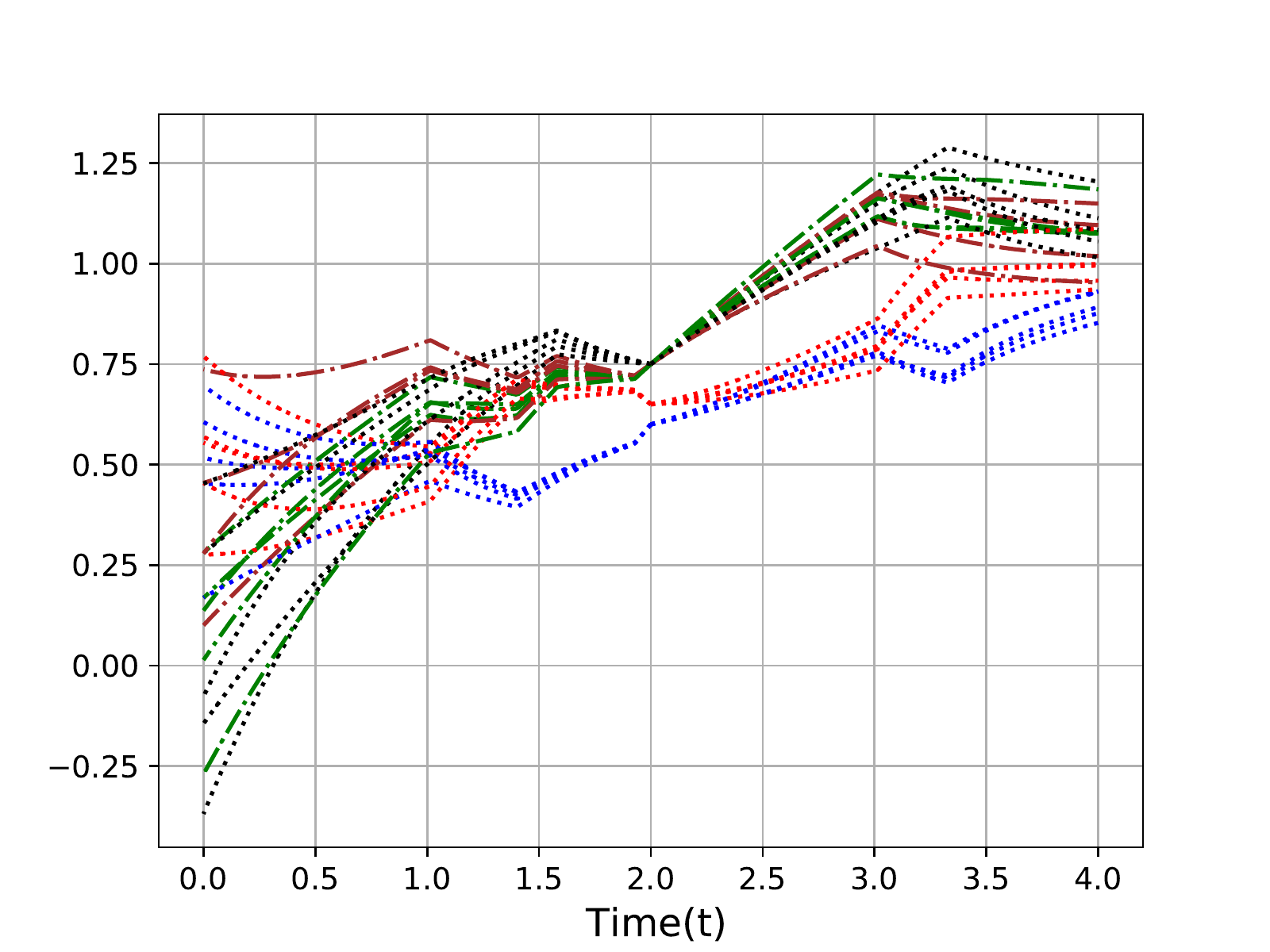}
                      \vspace{-20pt}
                     \caption{The time evolution of the opinions of the agents under the influence of the two control channels \(\cont_1\opt\) and \(\cont_2\opt\) are shown here.}
                    \label{fig:fig06}
            \end{minipage}
            \begin{minipage}{0.48\textwidth}
            \centering
                     \captionsetup{width=0.9\linewidth}
                     \includegraphics[width=1.1\linewidth]{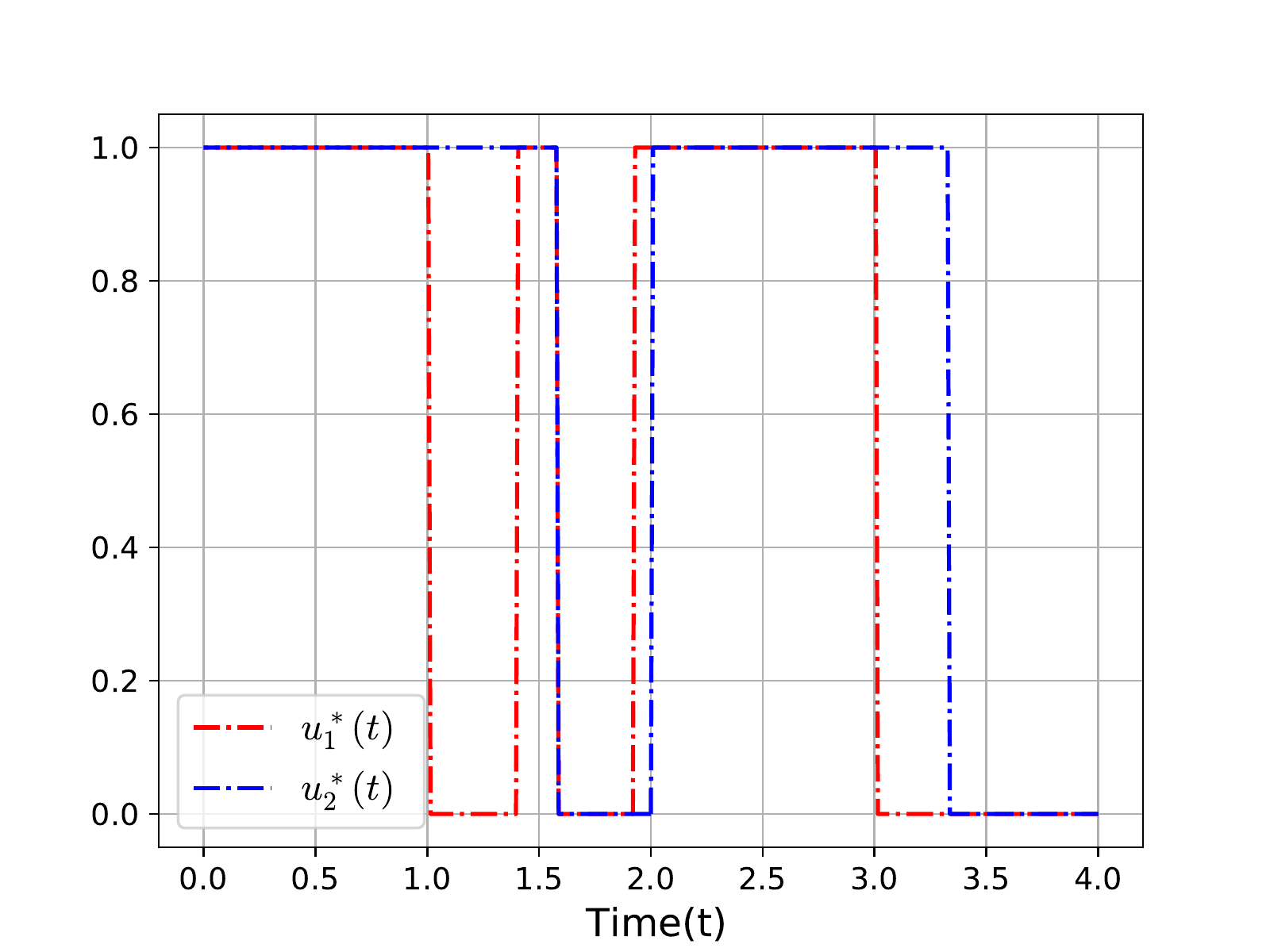}
                     \vspace{-20pt}
                    \caption{We depict the optimal budget allocations to the influence channels \(\cont_1\opt\) and \(\cont_2\opt\) under the influence channel matrix \(B\).}
                     \label{fig:fig07}
            \end{minipage}
       \end{figure}
       
       \begin{figure}[h]
     \centering
             \begin{minipage}{0.48\textwidth}
             \centering
                    \captionsetup{width=0.9\linewidth}
                    \includegraphics[width=1.1\linewidth]{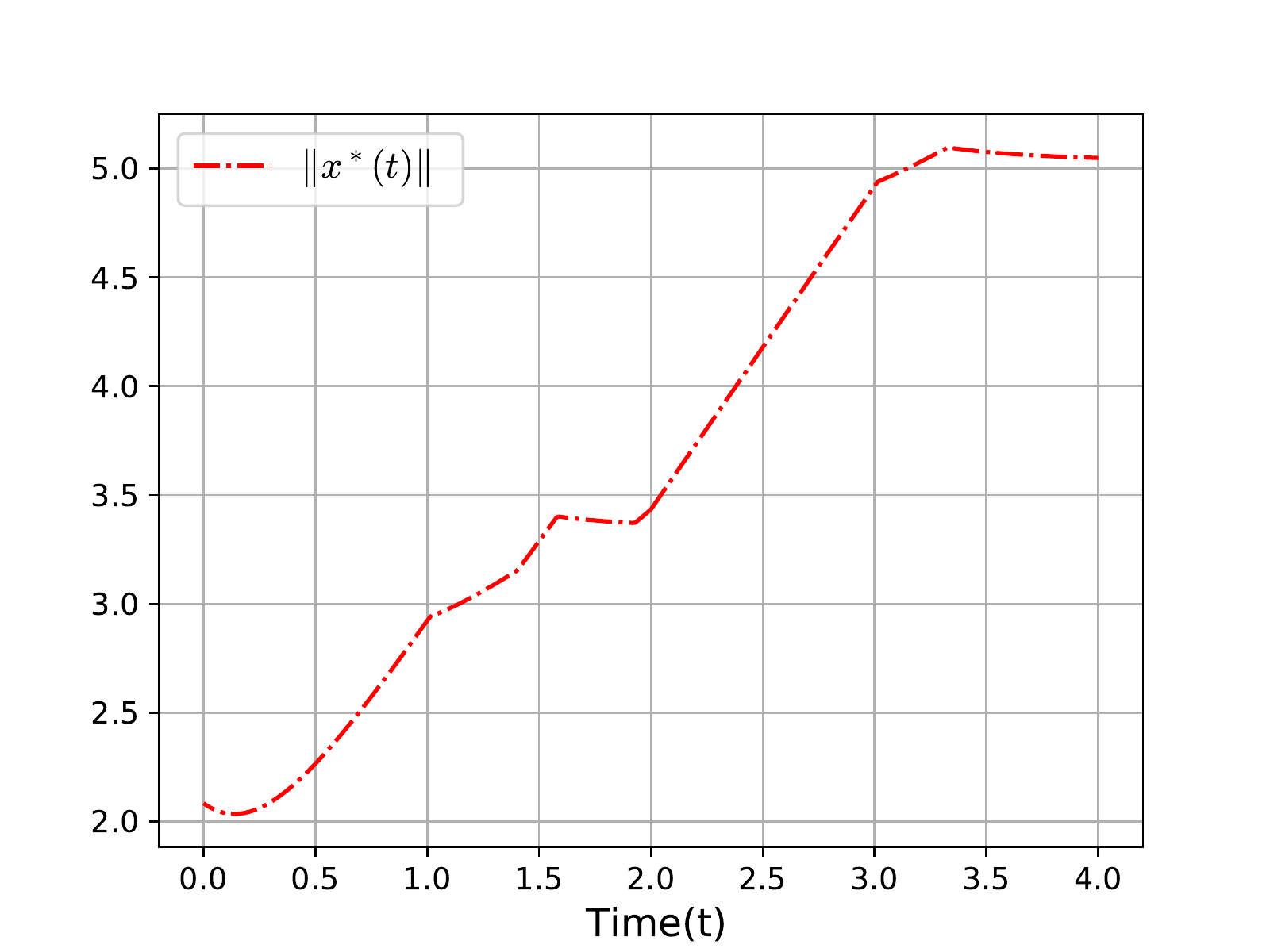}
                      \vspace{-20pt} 
                     \caption{The optimal time trajectory of the norm of the opinions of the agents are shown here; the figure shows non minimum phase behaviour e.g., around time 0s, 1.6s, and after time 3.4s.}
                    \label{fig:fig08}
            \end{minipage}
            \begin{minipage}{0.48\textwidth}
            \centering
                     \captionsetup{width=0.9\linewidth}
                      \includegraphics[width=1.1\linewidth]{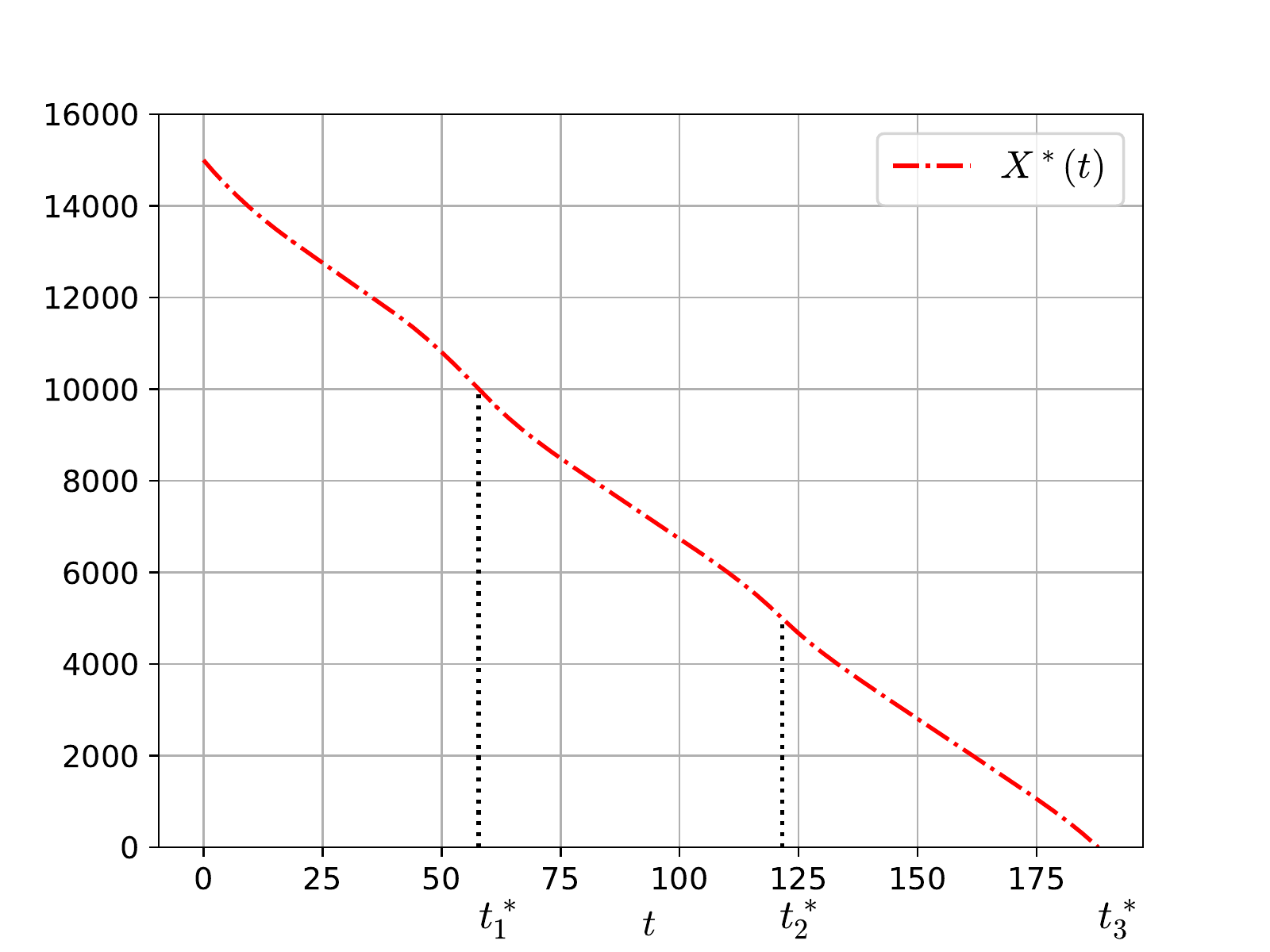}
                    \vspace{-20pt}
                   \caption{We plot the aircraft landing distance \(X^*(\cdot)\) from the runway, and \(t_1\opt, t_2\opt, t_3\opt\) mark the optimal intermediate time instants.}
                     \label{fig:fig01}
            \end{minipage}
       \end{figure}

 \begin{figure}[h!]
     \centering
             \begin{minipage}{0.48\textwidth}
             \centering
                   \captionsetup{width=0.9\linewidth}
                    \includegraphics[width=1.1\linewidth]{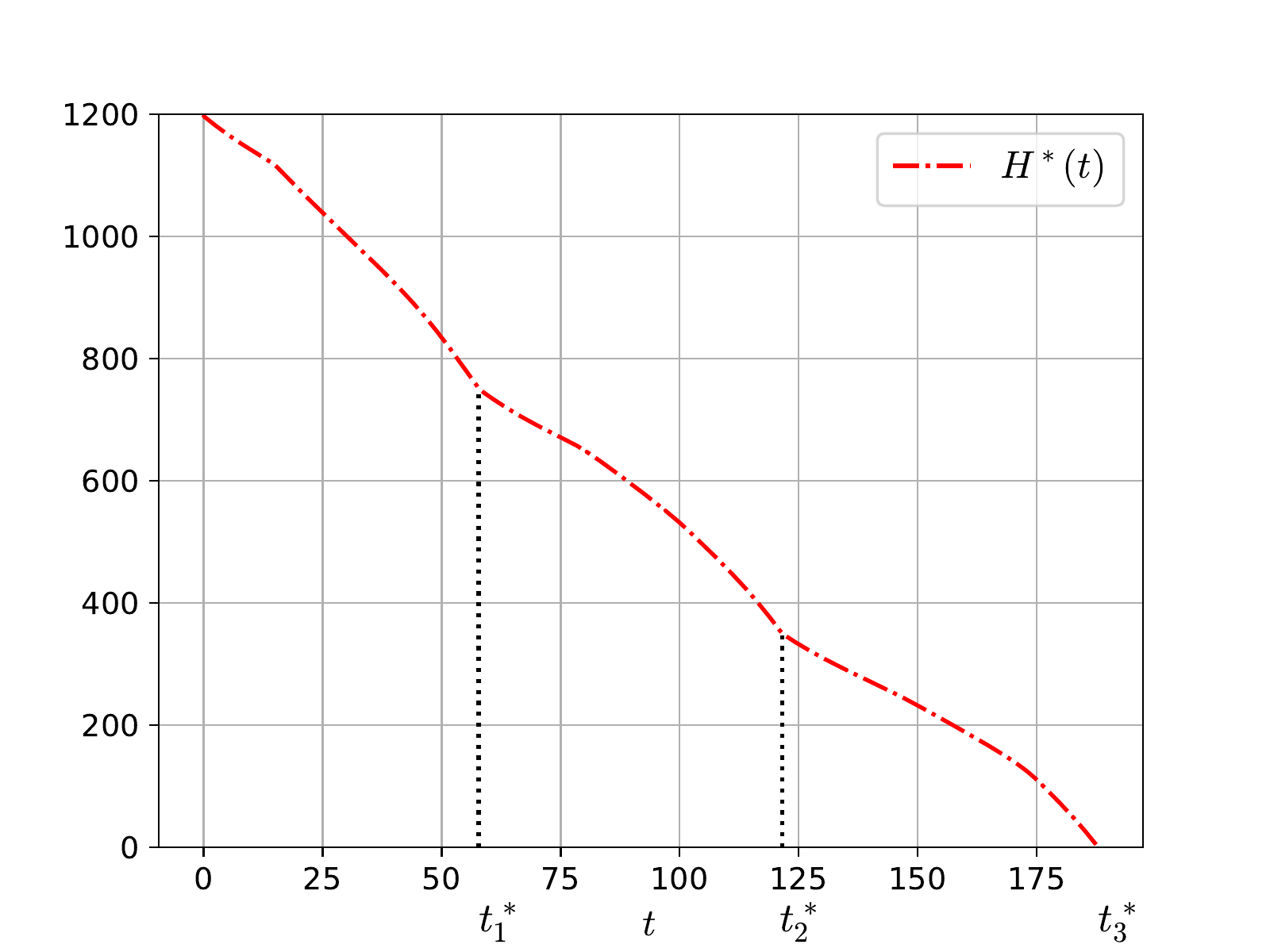}
                       \vspace{-20pt}
                      \caption{Illustraition of the optimal time evolution of the aircraft altitude \(H^*(\cdot)\), it can be seen that the intermediate constraints are active at \(t_1\opt, t_2\opt\),and \(t_3\opt\).} 
                    \label{fig:fig02}
            \end{minipage}
            \begin{minipage}{0.48\textwidth}
            \centering
                     \captionsetup{width=0.9\linewidth}
                      \includegraphics[width=1.1\linewidth]{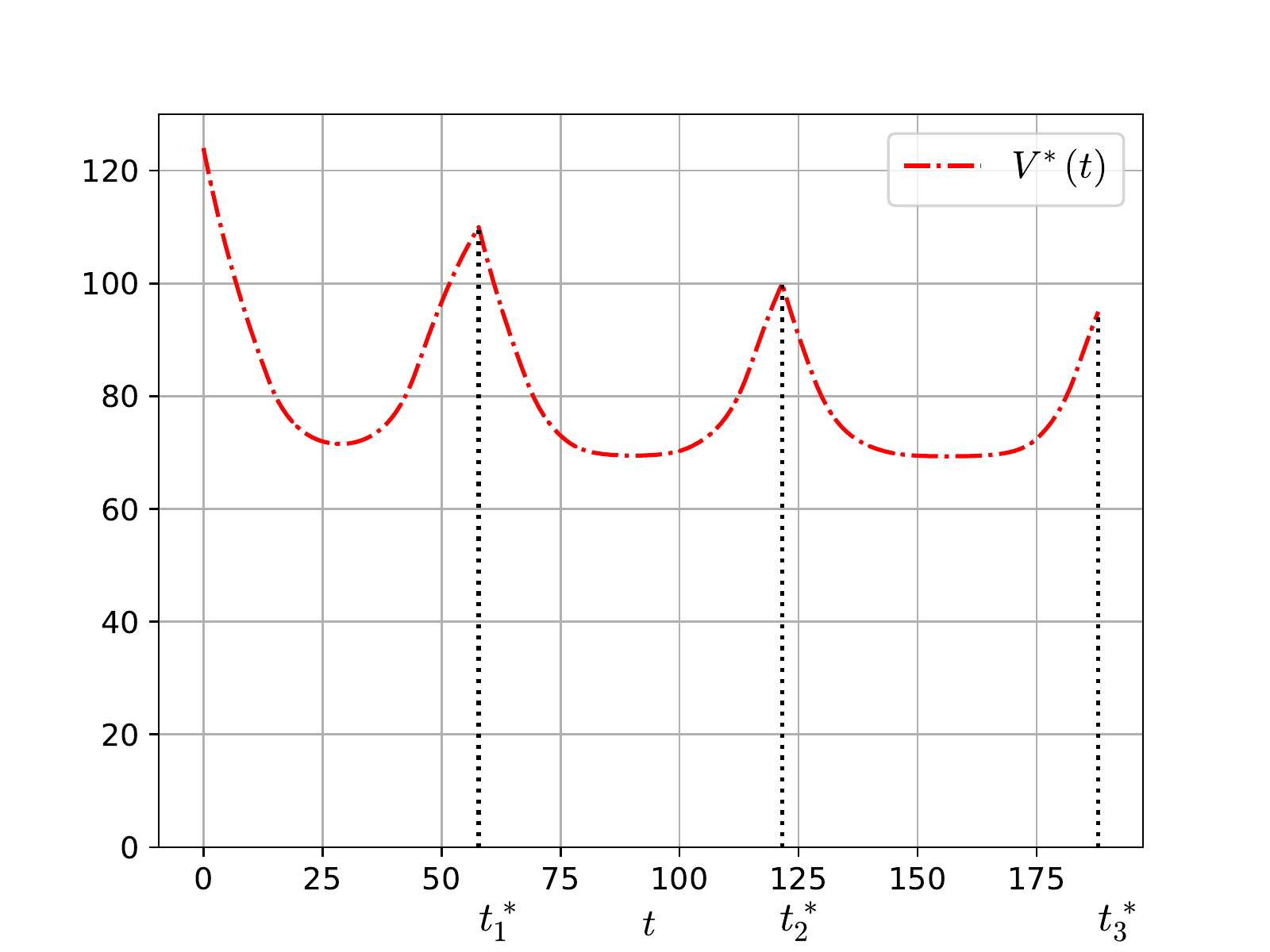}
                     \vspace{-20pt}
                    \caption{The optimal velocity profile \(V^*(\cdot)\) of the aircraft along with the active intermediate velocity constraints are shown here.}
                     \label{fig:fig03}
            \end{minipage}
       \end{figure}
           
       \begin{figure}[h]
     \centering
             \begin{minipage}{0.48\textwidth}
             \centering
                     \captionsetup{width=0.9\linewidth}
                     \includegraphics[width=1.1\linewidth]{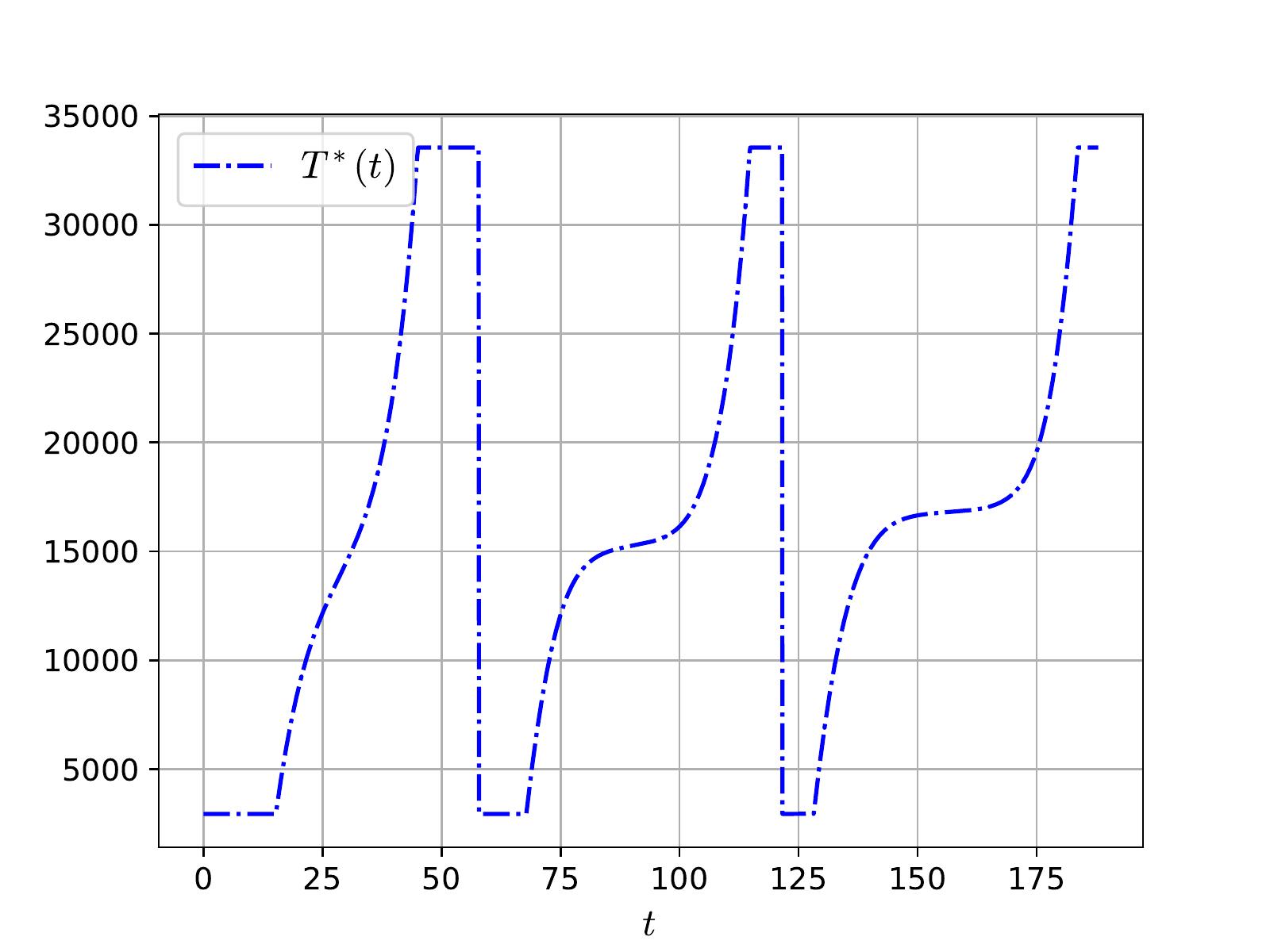}
                      \vspace{-20pt}
                     \caption{We depict the optimal thrust trajectory \(T\opt(\cdot)\) of the aircraft where the bounds on the thrust are seen to be active.}
                    \label{fig:fig04}
            \end{minipage}
            \begin{minipage}{0.48\textwidth}
            \centering
                     \captionsetup{width=0.9\linewidth}
                     \includegraphics[width=1.1\linewidth]{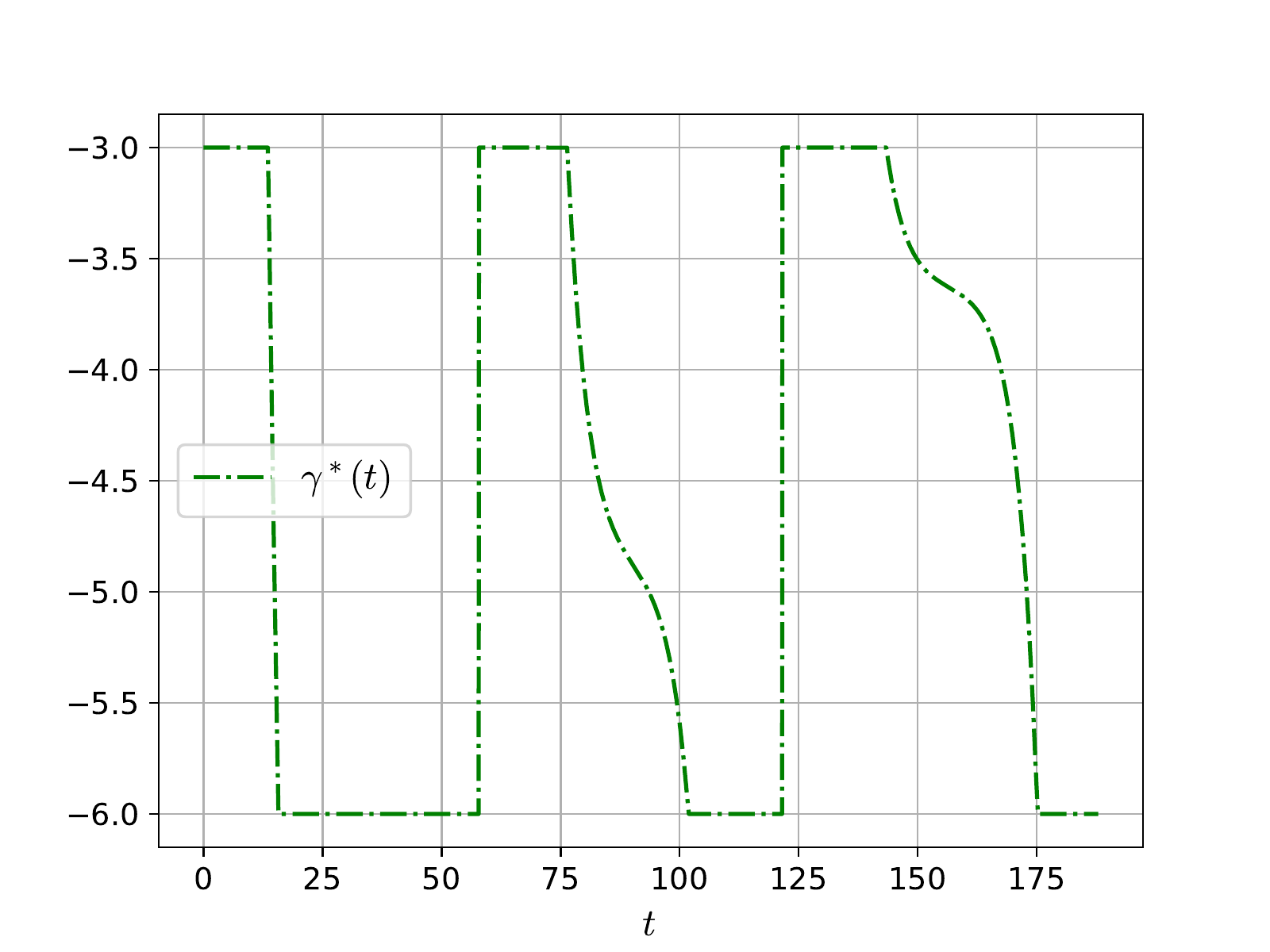}
                     \vspace{-20pt}
                    \caption{The optimal flight path angle trajectory \(\gamma\opt\) w.r.t. time, and with active control bounds \(-6\degree \leq \gamma\opt \leq -3\degree\) are observed here. }                         
                     \label{fig:fig05}
            \end{minipage}
       \end{figure}
    
\begin{example}
        Here we illustrate our results on an aircraft landing-approach problem. Consider an aircraft approaching a runway for landing and let the aircraft be at some distance \(X_0\) from the runway at a given height \(H_0\) at time \(t_0 = 0\) second.  For smooth landing, the aircraft is commanded to hit particular frames at several intermediate distances from the runway. This results in an optimal control problem of minimizing a weighted sum of fuel and energy with intermediate constraints featuring constraints on the states of the aircraft at these intermediate distances. We restrict our attention to the movement of the aircraft on the vertical plane, and describe its motion using a point mass model for simplicity. Further, we confine our investigation to a case where the flight path angle dynamics is negligible, see \cite{ref:Pie-85} for a discussion. Variants of this models with rectangular constraints are commonly used in air traffic control (ATC) research; see e.g., \cite{ref:Lygero-99}. The equations of motion of the aircraft are:
        \begin{equation} \label{s:eg1}
        \begin{cases}
          \begin{aligned}
              &       \frac{\dd V(t)}{\dd t} =  \frac{T(t)}{m} - g\sin(\gamma(t)) - \frac{1}{2m}\ro V(t)^2S(C_{D_0} + \eta C_L^2),\\
              &       \frac{\dd H(t)}{\dd t}    = V(t)\sin(\gamma(t)),\\
              &        \frac{\dd X(t)}{\dd t}    = V(t)\cos(\gamma(t)),\\
          \end{aligned}
          \end{cases}
        \end{equation} 
		where we denote the landing approach range by $X$, the aircraft speed by $V$, the flight path angle by $\gamma$, the altitude by $H$, and the thrust by $T$. Here \((X,V,H)\) are the state variables and \((T, \gamma)\) are the control inputs. Observe that the control variable \(\gamma\) enters the dynamics in a highly nonlinear fashion in \eqref{s:eg1}.
          
		   For the system \eqref{s:eg1} we consider minimizing a weighted sum of fuel and energy, which in mathematical terms is an \emph{\(L_2\)-regularized \(L_1\)-minimization} problem over the control, i.e., we minimize
\begin{equation}
	\J(\cont) \Let \lambda_1\norm{\cont}_1 + \frac{r}{2}\norm{\cont}_2^2,
\end{equation}
where \(\lambda_1, r > 0\) are positive weights, and \( t \mapsto \cont(t) \Let \pmat{T(t) \\ \gamma(t)}\) are the control inputs.
           For simulations we consider the data of the F-104G fighter aircraft from \cite{ref:Pie-85}, where $W= mg =$ \SI{7180}{\newton} is the weight of the aircraft, $g =$ \SI{9.81}{\meter\per\second\squared} is the acceleration due to gravity, $\rho =$ \SI{1.22625}{\kilogram\per\meter\cubed} is the air density, and $S = $\SI{18.2}{\meter\squared} is the aircraft planform area. The dimension-free aerodynamic constants $C_{D_0} = 0.198681$ and $\eta = 0.114738$. The bounds on the control actions are \( 300g \leq T \leq 3420g\) and \SI{-3}{\degree} \(\leq \gamma \leq\) \SI{-6}{\degree}. The initial time \(t_0 =\) \SI{0}{\second} is fixed, and the intermediate time instants \(t_1, t_2,\) and \(t_3\) are free. The state constraints at these time instants are \((X(t_0), H(t_0), V(t_0)) =\) (\SI{15}{\km}, \SI{1197}{\meter}, \SI{124}{\meter\per\second}), \(( X(t_1), H(t_1), V(t_1))=  (\SI{10}{\km}, \SI{750}{\meter} ,\SI{110}{\meter\per\second}),\) \( (X(t_2), H(t_2), V(t_2)) = (\SI{5}{\km}, \SI{350}{\meter} ,\SI{100}{\meter\per\second})\),  \((H(t_3), V(t_3),  X(t_3))\) \(= (\SI{0}{\km}, \SI{0}{\meter} ,\SI{90}{\meter\per\second}).\) The landing distance is measured in \SI{}{\km}, height in \SI{}{\meter} and velocity in meter-per-second \SI{}{\meter\per\second}.

          \textup{The optimal control described by Theorem \ref{t:OCP} and Algorithm \ref{a:Algo} along with the idea of Proposition \ref{p:TPBVP} are applied in simulation to the system \eqref{s:eg1} and the corresponding results are shown in Figures \ref{fig:fig01} through \ref{fig:fig05}. Figure \ref{fig:fig01} plots the evolution of landing distance from the runway: since the flight path angle is very small, the value of the landing distance \(X(t)\) largely depends on the velocity profile shown in Figure \ref{fig:fig03}. Also, the rate of approach towards runway is high whenever velocity is large. The optimal intermediate time instants are \(t_1\opt = \SI{57.8}{\s}, t_2\opt = \SI{121.5}{\s}, t_3\opt = \SI{188}{s}\), and the corresponding intermediate state constraints are active at these instants. Figure \ref{fig:fig02} shows the evolution of the height: the rate of descent largely depends on the velocity and the flight path angle and this rate is large whenever the control input \(\gamma\) is set to its lower limit. The minimum descent rate is \SI{4.23}{\meter\per\second} while the maximum is \SI{11.5}{\meter\per\second}, which is well within the tolerable range for a typical fighter aircraft. As mentioned above, in Figure \ref{fig:fig03} we plot the velocity of the aircraft; on each interval between successive intermediate constraints the optimality requirements set the thrust to its lower limit initially, resulting in a decrease in the aircraft velocity, but the terminal constraints on the velocity on each interval pulls the thrust towards its higher limit, resulting in an increase in its velocity. The time averaged velocity on the first interval \([\SI{0}{\s}, \SI{57.8}{\s}]\) is high and this quantity decreases on subsequent intervals, resulting in a higher time averaged rate of descent on the first interval.}

            \textup{Figure \ref{fig:fig04} plots the evolution of the thrust: the thrust profile depends mainly on the lift requirement, optimality criterion, and the velocity constraints. On each interval between successive intermediate constraints, the thrust initially stays at its lower limit in order to have minimum cost, but due to the intermediate constraints on the velocity, the thrust gradually increases after some time and reaches its maximum limit. Further, the overall thrust requirement decreases on subsequent intervals because of relaxation in the constraints on the velocity. In Figure \ref{fig:fig05} we plot the flight path angle: the evolution of the angle is mainly influenced by the descent and optimality requirements, and its magnitude follows similar trends as that of the thrust, i.e., on each interval, optimality forces it to its lower limit initially, while the intermediate constraints on the height pull it towards its higher limit. Finally, the overall requirement decreases on subsequent intervals (i.e., time duration for which \(\gamma\) is set to \(\SI{-3}{\degree}\) increases while the time duration of \(\SI{-6}{\degree}\) decreases on subsequent interval) because the level of descent required \((\SI{447}{\m}, \SI{400}{\m}, \SI{350}{\m})\) decreases progressively.}
\end{example}


	\bibliographystyle{alpha}
	\bibliography{ref}

\newcommand{\etalchar}[1]{$^{#1}$}
\begin{thebibliography}{CNQR16}

\bibitem[Bet98]{ref:Bet-98}
J.~T. Betts.
\newblock Survey of numerical methods for trajectory optimization.
\newblock {\em Journal of Guidance, Control, and Dynamics}, 21(2):193--207,
  1998.

\bibitem[Bor08]{ref:Borkar-08}
V.~S. Borkar.
\newblock {\em Stochastic {A}pproximation: a {D}ynamical {S}ystems
  {V}iewpoint}.
\newblock Hindustan Publishing Agency, New Dehli, India, 2008.

\bibitem[BOW11]{ref:BlaOnoWil-11}
L.~{Blackmore}, M.~{Ono}, and B.~C. {Williams}.
\newblock Chance-constrained optimal path planning with obstacles.
\newblock {\em IEEE Transactions on Robotics}, 27(6):1080--1094, Dec 2011.

\bibitem[Cla13]{ref:Cla-13}
F.~H. Clarke.
\newblock {\em Functional {A}nalysis, {C}alculus of {V}ariations and {O}ptimal
  {C}ontrol}, volume 264 of {\em Graduate Texts in Mathematics}.
\newblock Springer, London, 2013.

\bibitem[CNQR16]{ref:ChaNagQueRao-16}
D.~Chatterjee, M.~Nagahara, D.~E. Quevedo, and K.~S.~Mallikarjuna Rao.
\newblock Characterization of maximum hands-off control.
\newblock {\em Systems \& Control Letters}, 94:31--36, 2016.

\bibitem[Den37]{ref:CHden-37}
C.~H. Denbow.
\newblock {\em A generalized form of the problem of {Bolza}}.
\newblock ProQuest LLC, Ann Arbor, MI, 1937.
\newblock Thesis (Ph.D.)--The University of Chicago.

\bibitem[DK11]{ref:DmiKag-11}
A.~V. Dmitruk and A.~M. Kaganovich.
\newblock Maximum principle for optimal control problem with intermediate
  constraints.
\newblock {\em Computational Mathematics and Modeling}, 22(2):180--215, 2011.

\bibitem[EPS{\etalchar{+}}17]{ref:EPS-17}
S.~{Eshghi}, V.~M. {Preciado}, S.~{Sarkar}, S.~S. {Venkatesh}, Q.~{Zhao},
  R.~{D'Souza}, and A.~{Swami}.
\newblock Spread, then target, and advertise in waves: Optimal capital
  allocation across advertising channels.
\newblock In {\em 2017 Information Theory and Applications Workshop (ITA)},
  pages 1--10, Feb 2017.

\bibitem[Fil88]{ref:Fil-88}
A.~F. Filippov.
\newblock {\em Differential {E}quations with {D}iscontinuous {R}ighthand
  {S}ides}, volume~18 of {\em Mathematics and its Applications (Soviet
  Series)}.
\newblock Kluwer Academic Publishers Group, Dordrecht, 1988.
\newblock Translated from the Russian.

\bibitem[Hes08]{ref:Hesse-08}
H.~K. Hesse.
\newblock {\em Multiple shooting and mesh adaptation for PDE constrained
  optimization problems.}
\newblock PhD thesis, Ruprecht-Karls-University, Heidelberg, 2008.

\bibitem[KSC19]{ref:yogesh-19}
Y.~Kumar, S.~Srikant, and D.~Chatterjee.
\newblock Optimal multiplexing of sparse controllers for linear systems.
\newblock {\em Automatica}, 106:134 -- 142, 2019.

\bibitem[KY03]{ref:KusYin-03}
H.~J. Kushner and G.~G. Yin.
\newblock {\em Stochastic {A}pproximation and {R}ecursive {A}lgorithms and
  {A}pplications}.
\newblock Springer, New York, 2003.

\bibitem[Lib12]{ref:Lib-12}
D.~Liberzon.
\newblock {\em Calculus of {V}ariations and {O}ptimal {C}ontrol {T}heory}.
\newblock Princeton University Press, Princeton, NJ, 2012.
\newblock A concise introduction.

\bibitem[LTS99]{ref:Lygero-99}
J.~Lygeros, C.~Tomlin, and S.~Sastry.
\newblock Controllers for reachability specifications for hybrid systems.
\newblock {\em Automatica}, 35(3):349 -- 370, 1999.

\bibitem[NM14]{ref:NagMar-14}
M.~{Nagahara} and C.~F. {Martin}.
\newblock ${L^1}$ control theoretic smoothing splines.
\newblock {\em IEEE Signal Processing Letters}, 21(11):1394--1397, Nov 2014.

\bibitem[NMH12]{ref:NagMatHay-12}
M.~Nagahara, T.~Matsuda, and K.~Hayashi.
\newblock Compressive sampling for remote control systems.
\newblock {\em IEICE Transactions on Fundamentals of Electronics,
  Communications and Computer Sciences}, 95(4):713--722, 2012.

\bibitem[NQN16]{ref:NagQueNes-16}
M.~Nagahara, D.~E. Quevedo, and D.~Ne{\v s}i{\'c}.
\newblock Maximum hands-off control: a paradigm of control effort minimization.
\newblock {\em IEEE Transactions on Automatic Control}, 61(4), 2016.

\bibitem[Pie85]{ref:Pie-85}
B.~L. Pierson.
\newblock Optimal aircraft landing-approach trajectories: A comparison of two
  dynamic models.
\newblock {\em Annual Review in Automatic Programming}, 13:139 -- 145, 1985.
\newblock Control applications of nonlinear programming and optimization.

\bibitem[Rao10]{ref:Rao-10}
A.~Rao.
\newblock A survey of numerical methods for optimal control.
\newblock {\em Advances in the Astronautical Sciences}, 135, 01 2010.

\bibitem[RM51]{ref:RobMon-51}
H.~Robbins and S.~Monro.
\newblock A stochastic approximation method.
\newblock {\em Annals of Mathematical Statistics}, 22:400--407, 09 1951.

\bibitem[SC16]{ref:SriCha-16}
S.~Srikant and D.~Chatterjee.
\newblock A jammer's perspective of reachability and {LQ} optimal control.
\newblock {\em Automatica}, 70:295--302, 2016.

\bibitem[SEM00]{ref:SunEgeMar-00}
S.~{Sun}, M.~B. {Egerstedt}, and C.~F. {Martin}.
\newblock Control theoretic smoothing splines.
\newblock {\em IEEE Transactions on Automatic Control}, 45(12):2271--2279, Dec
  2000.

\bibitem[VO69]{ref:VolOst-69}
Y.~M. Volin and G.~M. Ostrovskii.
\newblock A maximum principle for discontinuous systems and its application to
  problems with phase constraints.
\newblock {\em Radiophysics and Quantum Electronics}, 12(11):1253--1263, Nov
  1969.

\bibitem[ZTC17]{ref:Emn-17}
J.~Zhu, E.~Tr{\'e}lat, and M.~Cerf.
\newblock Geometric optimal control and applications to aerospace.
\newblock {\em Pacific Journal of Mathematics for Industry}, 9(1):8, Jul 2017.

\end{thebibliography}

\bigskip
\bigskip
\end{document}